\documentclass[11pt, reqno]{amsart}
\usepackage{graphicx}
\usepackage{amsfonts,amsmath,amssymb}
\usepackage{xcolor}
\usepackage{float}
\usepackage[bookmarks=false]{hyperref}
\usepackage{theorem}
\usepackage{mathrsfs}
\usepackage{pdfsync}
\usepackage{subfig}
\usepackage{stmaryrd}
\newcounter{theorem}
\def\openthm#1#2{\refstepcounter{theorem}\bigskip

{\noindent\bf#1~\thetheorem\if#2!{. }\else{ (#2).}\fi}
\it}

\def\thmskip{}
\newenvironment{theorem}[1][!]{\openthm{Theorem}{#1}}{\thmskip}
\newenvironment{lemma}[1][!]{\openthm{Lemma}{#1}}{\thmskip}

\newcounter{remark}
\def\openrem#1#2{\refstepcounter{remark}\bigskip
{\noindent \it \bfseries#1~\theremark\if#2!{. }\else{ (#2). }\fi}}
\newenvironment{remark}[1][!]{\openrem{Remark}{#1}}{\qed}

\newcounter{algorithm}
\def\openalg#1#2{\refstepcounter{algorithm}\bigskip
{\noindent \it \bfseries#1~\thealgorithm\if#2!{. }\else{ (#2). }\fi}}
\newenvironment{algorithm}[1][!]{\openalg{Algorithm}{#1}}{\qed}

\newcounter{definition}
\def\opendef#1#2{\refstepcounter{definition}\bigskip
{\noindent \bf#1~\thedefinition\if#2!{. }\else{ (#2). }\fi}\it}

\def\R{\mathbb{R}}
\def\N{\mathbb{N}}
\def\Z{\mathbb{Z}}

\def\dx{\,{\rm d}x}

\def\dt{\,{\rm d}t}

\def\<{\langle}
\def\>{\rangle}

\def\argmin{{\rm argmin}}

\def\conv{{\rm conv}}

\def\eps{\varepsilon}
\def\Bs{\mathcal{B}}
\def\Es{\mathcal{E}}
\def\Etot{E}
\def\Us{\mathcal{U}}
\def\As{\mathcal{A}}

\def\Ds{\mathcal{D}}

\def\Ts{\mathcal{T}}
\def\Ps{\mathcal{P}}
\def\Ss{\mathcal{S}}
\def\Ys{\mathcal{Y}}
\def\Ks{\mathcal{K}}

\def\ub{\boldsymbol{u}}
\def\yb{\boldsymbol{y}}

\def\zb{\boldsymbol{z}}
\def\vb{\boldsymbol{v}}
\def\wb{\boldsymbol{w}}
\def\eb{\boldsymbol{e}}
\def\Fb{\boldsymbol{F}}
\def\fb{\boldsymbol{f}}
\def\gb{\boldsymbol{g}}
\def\Gb{\boldsymbol{G}}
\def\epsb{\boldsymbol{\eps}}
\def\thetab{\boldsymbol{\theta}}

\def\a{{\rm a}}
\def\qc{{\rm qc}}

\newcommand{\smfrac}[2]{{\textstyle \frac{#1}{#2}}}
\numberwithin{equation}{section}

\begin{document}

\title{A Posteriori Error Estimates for Energy-Based Quasicontinuum Approximations of a Periodic Chain}


\author{Hao Wang}
\address{Hao Wang, Oxford University Mathematical Institute,
  24-29 St Giles', Oxford, OX1 3LB, UK}
\email{wangh@maths.ox.ac.uk}

\date{\today}

\begin{abstract}
We present a posteriori error estimates for a recently developed
atomistic/continuum coupling method, the Consistent Energy-Based QC
Coupling method. 
The error estimate of the deformation gradient combines a residual estimate and an a posteriori stability
analysis. The residual is decomposed into the residual due to the
approximation of the stored energy and that due to the
approximation of the external force, and are bounded in negative
Sobolev norms. In addition, the error estimate of the total energy
using the error estimate of the deformation gradient is also
presented. Finally, numerical experiments are provided to illustrate
our analysis.
\end{abstract}

\maketitle

\section{Introduction}
Quasicontinuum (QC) methods, or in general atomistic/continuum
coupling methods, are a class of multiscale methods for coupling an
atomistic model of a solid with a continuum model. These methods have
been widely employed in computational nano-technology, where a fully
atomistic model will result in a prohibitive computational cost but an
exact configuration is required in a certain region of the
material. In this situation, atomistic model is applied in the region
which contains the defect core to retain certain accuracy, while continuum model is applied
in the far field to reduce the computational cost. 

A number of QC methods have been developed in the past decades and are
classified in two groups: energy-based coupling methods and
force-based coupling methods. Despite the fact that the force-based methods are
easy to implement and extend to higher dimensional cases, energy-based
methods have certain advantages. For example, the forces derived from
an energy potential are conservative which could leads to a faster
convergence rate in computation, and the energy of an atomistic system can also be a
quantity of interest in real application. However, consistent
energy-based coupling methods can be tedious and restrictive on the
shape of the coupling interface in more than one
dimension (see \cite{Shimokawa:2004, E:2005a}) and it was not until recent that a practical consistent energy-based coupling method
was created by Shapeev \cite{Shapeev2010a}, which is the Consistent Energy-Based QC
Coupling method that we analyze in the present paper.

A number of literature on the rigorous analysis of different QC methods have been proposed since the
first one by Lin \cite{Lin:2003a}. However, most of the analysis are
on the a priori error analysis, and only a few
are on the a posteriori error analysis. Arndt and Luskin give a posteriori error estimates for the
QC approximation of a Frenkel-Kontorova model \cite{Arndt2008a, Arndt2008b, Arndt2007a}. A goal-oriented
approach is used and error estimates on different quantity of
interests, each of which is essentially the difference
between the values of a linear functional at the atomistic solution and
the QC solution, are proposed. The estimates are decomposed into two
parts, one is used to correctly chose the atomistic region and another
is used to optimally choose the mesh in the continuum region. Serge et
al. \cite{Serge2007a} give error estimates, also through a goal-oriented approach, of the original energy-based QC
approximation, whose consistency is not guaranteed. Both of the above
works employ the technique of deriving and solving dual
problems as a result of the goal-oriented approach. Ortner and
S{\"u}li \cite{Ortner:2008a} derive an a posteriori error indicator for a global norm
through a similar approach as ours. However, the QC method analyzed
there does not contain an approximation of the stored energy which is
essentially different from the QC method we are interested.

The present paper provides the a posteriori
error analysis for the Consistent Energy-Based QC
Coupling method \cite{Shapeev2010a} for a one dimensional periodic chain with
nearest and next nearest neighbour interactions. The
formulation of the QC approximation has the feature that the finite
element nodes in the continuum region are not restricted to reside at
the atomistic positions, which creates the situation that interaction
bonds often cross the element boundaries, which is
common in two dimensional formulation. We then derive the
residual in negative Sobolev norms and then the a posteriori stability
constant as a function of the QC solution. The error estimator of the
deformation gradient in $L_2$-norm is then obtained by combining these
two analysis. In addition, we derive an error estimator for the total
energy difference by using that of the deformation gradient. It should
be remarked that though both of the error estimators are global
quantities, they consist of contributions from element. As a result,
an adaptive mesh refinement algorithm is developed and applied to a
problem that mimics the vacancy in the two dimensional case, and the
numerical results are presented.

\subsection{Outline}
In Section \ref{Sec:A_and_Q_Model}, we first formulate the atomistic
model through both a continuous approach, i.e., the deformation and the
displacement are considered as continuous functions on the reference
lattice, and a discrete approach, which is always taken in previous
literature. We then formulate the Consistent Energy-Based QC Coupling
method in one dimensional setting.

In Section \ref{Sec:Residual_Analysis}, we derive the residual
estimates for the Consistent Energy-Based QC Coupling method in a
negative Sobolev norm. The residual is split into two part, one is due
to the approximation of the stored energy and the other is due to the
approximation of the external force. 

In Section \ref{Sec:Stability}, we give the a posteriori stability
analysis.

In Section \ref{Sec:A_Posteriori_Error}, we combine the residual
estimate and the stability analysis to give the a posteriori
error estimate of the deformation gradient in $L_2$-norm and that of
the total energy.

In Section \ref{Numerics}, we present a numerical example to
complement our analysis.

\section{Model Problem and QC Approximation}
\label{Sec:A_and_Q_Model}
\subsection{Atomistic Model}
\label{subsec:AModel}
As opposed to taking only a discrete point of view in many QC researches, we use both continuous
functions and discretized vectors to denote the displacement and
the deformation. The reason for doing this is that the
Consistent Energy-Based QC coupling method, which we analyze in this paper,
is easily formulated through the continuous approach, while discrete
formulations could make the residual analysis of the external forces
much easier.

For an infinite reference lattice with atomistic spacing $\eps$, we
make the partition $\Ts^\eps = \{T^\eps_\ell\}_{\ell= -\infty}^{\infty}$ of the domain $\R$ such that $\R =
\cup_{\ell =-\infty}^\infty T^\eps_\ell$ and $T^\eps_\ell = [(\ell-1)\eps,
\ell \eps]$. We then define the displacement and deformation of this
infinite lattice to be continuous piecewise
linear functions $u$, $y \in \Ps_1(\Ts^\eps) \cap C^0(\R)$.  We use
$\ub$ and $\yb$ to denote the vectorizations of $u$ and $y$ such that
$u_\ell = u(\ell\eps)$ and $y_\ell = y(\ell\eps)$. We know that $u_\ell$ and
$y_\ell$ are the physical displacement and deformation of atom $\ell$
respectively. 

To avoid technical difficulties with boundaries, we apply periodic
boundary conditions. We rescale the problem so that there are $N \in
\N$ atoms in each period and $\eps = 1/N$, which implies that $u$ and
$y$ are 1-periodic functions and $\ub$ and $\yb$ are $N$-periodic vectors. We also impose a zero-mean
condition to the admissible space of displacements, which is defined
to be
\begin{equation}
\Us = \big\{ u \in \Ps_1(\Ts^\eps) \cap C^0(\R): u(x+1) = u(x) \text{ and }
  \int_0^1 u(x) \dx =0\big\}.
\label{Def:ACSpaceDisplace}
\end{equation}
The set of admisible deformations is given by
\begin{equation}
\Ys = \big\{ y \in \Ps_1(\Ts^\eps) \cap C^0(\R): y(x) = Fx +u(x), u \in \Us \big\},
\label{Def:ACSetDeform}
\end{equation}
where $F>0$ is a given macroscopic deformation gradient.

As we mentioned above, it is necessary in the analysis of the external
forces to employ the
discretization of the displacement and the
deformation. Therefore, by the relationship between $u, y$ and their
vectorizations $\ub,
\yb$, the discrete space of displacement and the admissible set of
deformation are defined by
\begin{equation}
\Us^\eps = \{ \ub \in \R^\Z: u_{\ell+N} = u_\ell, \eps\sum_{\ell =1}^Nu_\ell
= 0\},
\label{Def:ADSpaceDisplace}
\end{equation}
and
\begin{equation}
\Ys^\eps = \{ \yb \in \R^\Z: y_{\ell+N} = F \ell \eps+u_\ell, \ub \in
\Us^\eps\},
\label{Def:ADSpaceDeform}
\end{equation}
where the zero-mean condition on the displacements, i.e., $\eps\sum_{\ell =1}^Nu_\ell
= 0$ is obatined by applying the trapezoidal rule to evaluate the integration
$\int_0^1u(x) \dx$ with respect to the partition $\Ts^\eps$ and using the
periodicity of $u$.

For simplicity of analysis, we adopt a pair
interaction model and assume that only nearest neighbours and the
next-nearest neighbours interact. With a slight abuse of notation, the {\em stored atomistic energy}
(per period) of an admissible deformation is then given by
\begin{align}
\Es_\a(y) &:= \eps \sum_{\ell=1}^N
                   \phi\Big(\frac{y(\ell\eps)-y((\ell-1) \eps)}{\eps}\Big)+\eps
                   \sum_{\ell=1}^N
                   \phi\Big(\frac{y(\ell\eps)-y((\ell-2) \eps)}{\eps}\Big)  \nonumber\\
                   &= \eps \sum_{\ell=1}^N
                   \phi\Big(\frac{y_\ell-y_{\ell-1}}{\eps}\Big)+\eps
                   \sum_{\ell=1}^N
                   \phi\Big(\frac{y_\ell-y_{\ell-2}}{\eps}\Big) =:\Es_\a(\yb),
\label{Eq:AStoredEnergy}
\end{align}
where $\phi \in
C^3((0,+\infty))$ is a Lennard-Jones type interaction potential. We
assume that there exists $r_\ast >0$ such that $\phi$ is convex in $(0,
r_\ast)$ and concave in $(r_\ast, +\infty)$.

For the formulation of the external energy, we first define the linear nodal
interpolation operator $I_\eps: C^0(\R) \rightarrow \Ps_1(\Ts^\eps)
\cap C^0(\R)$ such that 
\begin{equation}
I_\eps g(\ell \eps) = g(\ell \eps) \quad \forall g \in C^0(\R).
\end{equation}
Then given a dead load $f \in \Us$, we define the {\em external
  energy} (per period) caused by $f$ to be
\begin{equation}
\<f, u\>_\eps := \int_0^1 I_\eps(fu) dx = \sum_{\ell=1}^N \eps f_\ell
u_\ell =:  \<\fb, \ub\>_\eps,
\end{equation}
where $\fb$ and $\ub$ are the vectorizations of the external force $f$
and the displacement $u$ according to $\Ts^\eps$.

Thus, the {\em total energy} (per period) under
a deformation $y \in \Ys$ is given by
\begin{displaymath}
  \Etot_\a(y;F) = \Es_\a(y) - \<f, u\>_\eps,
\end{displaymath}
as $u$ is determined by $y$ and $F$. However, in our analysis, we
always assume that $F$ is given and as a result, we simply write $
\Etot_\a(y;F)$ as $  \Etot_\a(y)$.

The problem we wish to solve is to find 
\begin{equation}
  y_\a \in \argmin \Etot_\a(\Ys),
\label{Eq:LocalMinimaA}
\end{equation}
where $\argmin$ denotes the set 
of local minimizers.

\subsection{Notation of Partitions, Norms and Discrete Derivatives}
Though it is natural to introduce the QC approximation after the
atomistic model, we decide to pause here and introduce some important
notation that are used throughout the paper in order to make the
flow of the paper more smooth and save some space.

In Section \ref{subsec:AModel}, we have introduced the partition $\Ts^\eps$ of
the domain $\R$. We now fix the notation for a generalized partition.

Let $\Ts^{m} = \{T_k^{m}\}_{k = -\infty}^{\infty}$ be a given
partition such that $T_k^{m} = [x^{m}_{k-1}, x^{m}_{k}]$, where
$x^{m}_k > x^{m}_{k-1}$ are the nodes of the partition. We
denote the size (or the length) of the $k$'th element by $\eps^{m}_k :=
|T_k^{m}| = x^{m}_k - x^{m}_{k-1}$. We also define the mesh size vector $\epsb^{m}$ such
that $\epsb^{m}:= (\eps^{m}_k)_{k=-\infty}^{\infty} \in (\R^+)^\Z$.

Given a partition $\Ts^{m}$ and a function $g \in C^0(R)$, we define
the $\Ps_1$ direct interpolation $I_m : C^0(\R) \rightarrow
\Ps_1(\Ts^m) \cap C^0(\R)$ by 
\begin{equation}
\label{Eq:CLinNodInterp}
(I_mg)(x^m_i)  = g(x^m_i) \quad \forall g \in C^0(\R),
\end{equation}
and $I_mg$ is often denoted by $g_m$. We also denote the vectorization
of $g \in C^0(\R)$ with respect to $\Ts^m$ by $\gb^m$ such that
\begin{equation}
\label{Eq:ALinNodInterp}
g^m_j = g(x^m_j).
\end{equation}

Let $\mathcal{D}$ be a subset of $\Z$. For a vector $\vb \in \R^\Z$
and a partition $\Ts^{m}$, we define the (semi-)norms
\begin{align*}
\Vert \vb \Vert _{\ell^{p}_{\epsb^{m}}(\Ds)} = 
\left\{
\begin{array}{l l}
\Big(\sum_{\ell \in \mathcal{D}} \eps^{m}_\ell|v_\ell|^p
\Big)^{1/p},  & 1 \le p < \infty,\\
\max_{\ell \in \mathcal{D}} |v_\ell |,  & p = \infty.
\end{array} \right.
\end{align*}
In particular, if $n_{m}$ is the number of the nodes of
$\Ts^{m}$ that are in $[0,1]$, we simply define
\begin{align*}
\Vert \vb \Vert _{\ell^{p}_{\epsb^{m}}} = 
\left\{
\begin{array}{l l}
\Big(\sum_{\ell=1}^{n_{m}} \eps^{m}_\ell |v_\ell|^p
\Big)^{1/p},  & 1 \le p < \infty,\\
\max_{\ell = 1, \ldots, n_{m}} |v_\ell |,  & p = \infty.
\end{array} \right.
\end{align*}

We now define discrete derivatives. Suppose $v \in
C^0(\R)$ and $\vb^m$ is its vectorization according to
$\Ts^{m}$. We define the first and second order discrete derivative
${\vb^{m}}'$ by
\begin{equation}
{v^{m}}'_k = \frac{v^{m}_j-v^{m}_{j-1}}{\eps^{m}_j}, \text{ and },
{v^{m}}''_k = \frac{{v^{m}}'_{j+1}-{v^{m}}'_j}{\bar{\eps}^{m}_j},
\end{equation}
where $\bar{\eps}^{m}_j:=\smfrac{1}{2}(\eps^h_j+\eps^h_{j+1})$.

It can be proved that for $v^{m} \in
C^0(\R) \cap \Ps_1(\Ts^{m})$ and $\vb^{m}$ being its
vectorization, we have the identity
\begin{equation}
\label{Eq:FirstDerivativeL2Identity}
\| {v^{m}}' \|_{L^p[0,1]} = \Vert {\vb^{m}}' \Vert _{\ell^{p}_{\epsb^{m}}}.
\end{equation}

Since $\Ts^\eps$ is special and uniform, we simply use
$\epsb$ and $\eps$ to denote its mesh size vector and mesh size
without superscripts and subscripts.

In addition to these, we denote the left and the right limit of an open interval $\omega$ by $L_\omega$ and $R_\omega$, which are also used later in our analysis.

\subsection{QC Approximation}
\label{subsec:QCApprox}
The QC approximation we analyze in this paper is essentially the
Consistent Atomistic/Continuum Coupling method developed in
\cite{Shapeev2010a}. We briefly redevelop this method in 1D so
that it is easily understood and enough for us to carry out the
analysis.

We first decompose the reference lattice, which occupies $\R$, into an
atomistic region $\Omega_\a$, which should contain any 'defects', and a continuum region
$\Omega_c$, where the solution is expected to be smooth. Moreover, we
assume $\Omega_\a$ to be a union of open intervals and $\Omega_c$ to be
a union of closed intervals, and $\Omega_\a \cup
\Omega_c = \R$. Since we impose periodic boundary conditions on the
displacement and notationally it is easier to assume the
atomistic region is away from the boundary of the period we analyze,
we make the following assumptions on $\Omega_\a$ and $\Omega_c$:
\begin{itemize}
\item $\Omega_\a$ and $\Omega_c$ appear periodically with exactly
  period of $1$, i.e, if $x \in \Omega_\a$ then $x+1 \in \Omega_\a$ and the
  same for $\Omega_c$.
\item $\exists \delta>2\eps$ such that $\big((0,\delta) \cup (1-\delta,1)\big)
  \subset \Omega_c$, i.e., the atomistic region is contained in the
  'middle' of the chain.
\end{itemize}
Note that, there is no such a restriction, that the interfaces where different
regions meet should lie on the positions of the atoms, i.e., it is not
necessary that $\Omega_c \cap
\Omega_\a \in \eps \Z$, which was always assumed in
previous modeling and analysis of 1D QC method.

Then, in order to reduce the number of degrees of freedom, we make the
partition $\Ts^h = \{ T^h_k\}_{k=-\infty}^{\infty}$ of the domain $\R$
according to the above region decomposition of
$\R$ as follows:
\begin{itemize}
\item $T_k^h = [x^h_k,x^h_{k-1}]$ and $T_{k+K} = [x^h_k+1, x^h_{k-1}+1] =
  [x^h_{k+K}, x^h_{k-1+K}]$, which implies that the partition is
  $K$-periodic with $|\cup_{k=1}^KT_k|=1$, 
  and there are $K$ elements in each $[0,1]$. We also assume that
  $x^h_1$ is the left most node and $x^h_K$ is the right most
  node in $[0,1]$.
\item If $\ell \eps \in \Omega_\a$, then $\exists i \in \Z$ such that
  $x_i=\ell \eps$, i.e., every position of an atom in the atomistic region
  is a node of this partition.
\item $\partial \Omega_c$ is a node in this partition which means that
  each element is contained in only one of the two regions.
\item $|T_k^h| = \eps^h_k \ge 2\eps$ if $T^h_k \subset \Omega_c$, i.e., the
  size of each element in the continnum region is larger than or equal
  to $2\eps$.
\end{itemize}

We emphasize two definitions
\begin{equation}
\label{Def:lkandthetak}
\ell_k := \max_\ell \{\ell: \ell \eps \le x^h_k \}  \text{ and }
  \theta_k := \frac{x^h_k - \ell_k \eps}{\eps}, 
\end{equation}
which are extensively used in the analysis and significantly simply the
notation. Note that $0 \le \theta_k \le 1$.

Based on this partition of the domain, the QC space of displacement and the QC set of admissible
deformation are defined by
\begin{equation}
\label{Def:QCCSpaceDisplace}
\Us_{\qc} = \big\{u \in \Ps_1(\Ts^h) \cap C^0(\R) : u(x+1) = u(x) \text{ and } \int_0^1u(x) \dx=0\big\},
\end{equation}
and
\begin{equation}
\label{Def:QCCSetDeform}
\Ys_{\qc} = \big\{y \in  \Ps_1(\Ts^h) \cap C^0(\R) : y(x) = Fx+u(x), u
\in \Us_{\qc} \big\}.
\end{equation}

The discrete QC
space of displacement and the QC set of admissible deformation are
defined by
\begin{equation}
\label{Def:QCDSpaceDisplace}
\Us^h_{\qc} = \big\{\ub^h \in \R^\Z : u^h_k = u^h_{k+K}, \forall k \in \Z,
\text{ and } \sum_{k=1}^K\frac{1}{2}(x^h_{k+1}-x^h_{k-1})u^h_k=0\big\},
\end{equation}
and
\begin{equation}
\label{Def:QCDSpaceDeform}
\Ys^h_{\qc} = \big\{\yb^h \in \R^\Z : y^h_k = Fx_k+u^h_k, \ub \in
\Us^h_{\qc} \big\}.
\end{equation}
Note that unlike $\Us^\eps$ and $\Ys^\eps$, in which every vector has
the physical displacements and deformations of the atoms as its components,
$\Us^h$ and $\Ys^h$ only contain vectors whose components are the
values of displacements and deformations at the nodes of $\Ts^h$.

The approach to couple the atomistic and continuum energy is to
associate the energy with interaction bonds. The term {\it bond}
between atoms $i \in \Z$ and $i+r \in \Z$ refer to
the open interval $b = (i\eps,(i+r)\eps)$. In our case, since only nearest
neighbour and next nearest neighbour bonds are taken into account, $r
=1, 2$ only. 

To develop the coupling method, we define the operator $D_\omega y$ for an open interval $\omega
= (L_\omega, R_\omega) \subset \R$ and $y \in C^0(\R)$ such that
\begin{equation}
D_\omega y:= \frac{1}{|\omega|} \big(y(R_\omega)-y(L_\omega)\big).
\end{equation}

If we take any $y \in C^0(\R) \cap W^{1,\infty}(\R)$ as a deformation
(note that this 'deformation' might be non-physical) and a bond $b =
(i\eps,(i+r_b)\eps)$, we can define the atomistic energy contribution
of bond $b$ to the stored energy to be 
\begin{equation}
a_b(y) = \frac{|b \cap \Omega_\a|}{r_b} \phi\big(r_b D_{b \cap
  \Omega_\a} y \big),
\end{equation}
and its continuum energy contribution to the stored energy to be 
\begin{equation}
c_b(y) = \frac{1}{r_b} \int_{b\cap \Omega_c} \phi(\nabla_{r_b}y(x))\dx,
\end{equation}
where $\nabla_{r_b}y = r_by'(x)$.

Since we are only interested in the situation in $[0,1]$, which is
extended periodically to the whole domain, the set of bonds that we will consider is
\begin{equation}
\mathcal{B} = \big\{ (i\eps, (i+r)\eps): r =1,2, i =0,1,\ldots,N-1 \big\}.
\end{equation}

Therefore, coupling the two energy contributions together, the {\em stored QC energy}
(per period) of a deformation $y \in C^0(\R) \cap W^{1,\infty}(\R)$ is then given by
\begin{equation}
\Es_{\qc}(y) = \sum_{b \in B} \big[ a_b(y)+c_b(y) \big],
\label{QCStoredEnergy-1}
\end{equation}
which was shown in \cite{Shapeev2010a} to be a consistent coupling
method, where the definition of  consistency is as follows:
\begin{equation}
\Es'_\a(Fx)[v] = \Es'_{\qc}(Fx)[v]=0 \quad \forall v \in C^0(\R) \cap W^{1,\infty}(\R).
\end{equation}

Given a dead load $f \in \Us$ the QC approximation of the {\em external energy} (per period) caused by $f$ is given by
\begin{equation}
\<f,u_h\>_h :=  \int_0^1 I_h (f u_h) \dx =\sum_{k = 1}^K \frac{1}{2} (x^h_{k+1}-x^h_{k-1})
f^h_ku^h_k = : \<\fb^h, \ub^h\>_h,
\end{equation}
where $I_h$ is the linear nodal interpolation with respect to $\Ts^h$
, and $\fb^h$ and $\ub^h$ are the vectorizations of
$f$ and $u_h$.

Thus, the {\em total energy} (per period) of
a deformation $y_h \in \Ys_{\qc}$ is given by
\begin{displaymath}
  \Etot_{\qc}(y_h;F) = \Es_{\qc}(y_h) - \<f, u_h\>_h.
\end{displaymath}
For the same reason that $F$ is given, we write $\Etot_{\qc}(y_h;F)$ as $\Etot_{\qc}(y_h)$.
The problem we wish to solve is to find
\begin{equation}
\label{LocalMinimumQC}
  y_{\qc} \in \argmin \Etot_{\qc}(\Ys_{\qc}).
\end{equation}

\section{Residual Analysis}
\label{Sec:Residual_Analysis}
In this section, we bound the residual in a negative Sobolev norms. We
equip the space $\Us$ with the Sobolev norm
\begin{displaymath}
\|v\|_{\Us^{1,2}} = \Vert v' \Vert _{L^2[0,1]}, \quad \text{ for } v \in \Us,
\end{displaymath}
and denote it by
$\Us^{1,2}$. The norm on the dual $\Us^{-1,2} :=
(\Us^{1,2})^\ast$ is defined by
\begin{displaymath}
\| T \|_{\Us^{-1,2}} := \sup_{\substack{v \in \Us \\ \| v
    \|_{\Us^{1,2}}=1 }} T[v], \quad \text{ for } T \in \Us^{-1,2}. 
\end{displaymath}

In the following sections, we formulate the problems in variational
forms and then analyze the residual. 

\subsection{Variational Formulation and Residual}
Let $y_\a$ be a solution of the atomistic problem
\eqref{Eq:LocalMinimaA}. If ${y_\a}'(x)>0$ on $[0,1]$, $\Es_\a(y)$ has
the variational derivative at $y_\a$ and therefore, the first order
optimality condition for \eqref{Eq:LocalMinimaA} in variational form is
\begin{equation}
 \label{VariationalA}
  \Es'_\a(y_\a)[v] = \<f,v\>_\eps \qquad \forall v \in \Us,
\end{equation} 
where 
\begin{equation}
 \label{VariationalAStore}
 \Es'_\a(y_\a)[v] = \eps \sum_{b
  \in \Bs} \phi'(r_b D_b y_\a) r_bD_b v.
\end{equation}

Let $y_{\qc}$ be a solution of the QC problem \eqref{LocalMinimumQC}. If ${y_{\qc}}'(x)>0$ on $[0,1]$, then $\Es_{\qc}(y)$ has
the variational derivative at $y_{\qc}$ and the first order
optimality condition for \eqref{LocalMinimumQC} in variational form
is
\begin{equation}
\label{VariationalQC}
\Es'_{\qc}(y_{\qc})[v_h] = \<f,v_h\>_h \quad \forall v_h \in \Us_{\qc},
\end{equation}
where
\begin{align}
\label{VariationalQCStore}
\Es'_{\qc}(y_h)[v_h] &= \sum_{b \in B} \big[
a'_b(y_h)[v_h]+c_b(y_h)[v_h] \big] \nonumber \\
&= |b\cap \Omega_\a| \phi'(r_b D_{b \cap \Omega_\a} y_h) D_{b \cap
  \Omega_\a} v_h  + \sum_{b \in \Bs} \frac{1}{r_b} \int_{b \cap
  \Omega_c} \phi'(\nabla_{r_b} y_h) \nabla_{r_b} v_h \dx \quad \forall
v_h \in \Us_{\qc}.
\end{align}

In conforming finite element analysis, where the finite element
solution space is a subspace of the original solution space, the residual is defined as the
quantity we obtain by inserting the computed solution to the equation
which the real solution satisfies. However, in our case, $\Ys_{\qc}$ is
in general not a subspace of $\Ys_\a$, and hence the functional $\Es_\a(\cdot)$ is not defined on
$\Ys_{\qc}$ in general and $\Es_{\qc}(\cdot)$ is not defined on
$\Ys_\a$ either. The way through which we circumvent this
difficulty is to define mappings between the solution spaces so that the residual could be well defined. In
concrete, we define $J_\Us : \Us \rightarrow
\Us_{\qc}$ and $J_{\Us_{\qc}} : \Us_{\qc} \rightarrow \Us$ such that 
\begin{equation}
\label{definitionJv}
J_{\Us} u = I_h u - \frac{1}{2} \sum_{\ell=1}^{K}(x^h_{k+1} -
x^h_{k-1})  u(x^h_k) \quad \forall u \in \Us,
\end{equation}
and
\begin{equation}
\label{definitionJvh}
J_{\Us_{\qc}} u_h = I_\eps u_h - \eps \sum_{\ell=1}^{N} u_h(\ell \eps)
\quad \forall u_h \in \Us_{\qc}.
\end{equation}
It is easy to check that $J_{\Us} u$ and $J_{\Us_{\qc}} u_h$ satisfy
the corresponding mean zero condition of $\Us$ and $\Us_{\qc}$, which implies that $J_\Us u \in \Us_{\qc}$ and $J_{\Us_{\qc}}
u_h \in \Us$. With a slight abuse of notation, we define
\begin{equation}
\label{definitionJy}
J_{\Us} y = Fx + J_{\Us} u=Fx+I_h u - \frac{1}{2} \sum_{\ell=1}^{K}(x^h_{k+1} -
x^h_{k-1})  u(x^h_k) = I_hy - \frac{1}{2} \sum_{\ell=1}^{K}(x^h_{k+1} -
x^h_{k-1})  u(x^h_k) \quad \forall y \in \Us ,
\end{equation}
and
\begin{equation}
\label{definitionJhy}
J_{\Us_{\qc}} y_h = Fx + J_{\Us_{\qc}} u_h=Fx+I_\eps u_h - \frac{1}{2}
\eps \sum_{\ell=1}^{N} u_h(\ell \eps)=I_\eps y - \frac{1}{2}
\eps \sum_{\ell=1}^{N} u_h(\ell \eps)  \quad \forall y_h \in \Us_{\qc}.
\end{equation}

We then define the residual (at the solution $y_{\qc}$) to be
\begin{align}
\label{definitionTotalResidual}
R[v] &= E'_\a(J_{\Us_{\qc}}y_{\qc})[v] \nonumber\\
       &= E'_\a(J_{\Us_{\qc}}y_{\qc})[v] - E'_{\qc}(y_{\qc})[v] \nonumber\\
       &= \big[\Es'_\a (J_{\Us_{\qc}}y_{\qc})[v] - \<f,v\>_\eps\big]
       -\big[\Es'_{\qc}(y_{\qc})[J_\Us v]-\<f,J_\Us v\>_h\big]
       \nonumber\\
       &=  \big[\Es'_\a (J_{\Us_{\qc}}y_{\qc})[v]-\Es'_{\qc}(y_{\qc})[J_\Us
       v]\big] + \big[\<f,J_\Us v\>_h-\<f,v\>_\eps  \big].
\end{align}
and understand $R$ as a functional in $\Us^{-1,2}$. By this formulation, we essentially split the residual into two parts: the first part
is the residual of the stored energy and the second part is the
residual of the external force. We will bound these two parts in
the following sections.

\subsection{Estimate of the Residual of the Stored Energy}
\label{SectionResidualStoredEnergy}
In this section, we analyze the first part of
\eqref{definitionTotalResidual}, which is the residual of the stored
energy. 

Before we give the theorem, we make several definitions that simplify our notation.

First, we define the set $\Ks_c$ to be
\begin{equation}
\label{Set:K_c}
\Ks_c := \big\{ k: k \in \{1,\ldots,K\} \text{ such that } T_k \cap
[0,1] \neq \emptyset \text{ but } T_k \cap (1,+\infty) = \emptyset \text{ and } T_k \subset \Omega_c \big\},
\end{equation}
which is essentailly the set of indices of the elements in the continuum region in $[0,1]$.

Second, suppose the atomistic region consists of $M$ disjoint subregions
in $[0,1]$, i.e., $\Omega_\a \cap [0,1] = \cup_{i=1}^{M} \Omega^i_\a$ among which
$\Omega^i_\a\cap \Omega^j_\a = \emptyset$ if $i \neq j$, we define the
nodes lie on the atomistic-continuum interface of
the atomistic regions be $x^h_{La_i}$, $i = 1, \ldots, M$ and those lie
on the right interface be $x^h_{Rc_i}$, $i=1, \ldots, M$.

Third, we define $\Ks'_c \subset \Ks_c$ to be the set of indices of the elements in the continuum 
region but not adjacent to an atomistic region, i.e., $\forall k \in \Ks'_c$, $k \ne L_{a_i}$ and $k-1 \ne R_{a_i}$, $\forall i \in \{1,2,\ldots,M\}$.

Using these definitions, we have the following theorem.

\begin{theorem}
\label{Theo:ResidualStoredEnergy}
For $y_h
\in \Ys$ with $y_h'(x) >0$, we have 
\begin{equation}
\label{ResidualStoredEnergy}
\big\|\Es'_\a(J_{\Us_{\qc}} y_h)[\cdot]-\Es'_{\qc}(y_h)[J_\Us\cdot ]\big\|_{\Us^{-1,2}} \le
\big\{ \sum_{k \in \Ks_c} {\eta_k^e}^2 \big\}^{\frac{1}{2}} =: \mathscr{E}_{\rm store}(y_h),
\end{equation}
where 
\begin{equation}
\eta_k^e = \big( \frac{1}{2}\sum_{j = 0}^{2} [[\phi'
]]_{\ell_{k-1}+j}^2 + \frac{1}{2}\sum_{j = 0}^{2} [[\phi'
]]_{\ell_{k}+j}^2 \big)^{\frac{1}{2}},
\end{equation}
if $k \in \Ks'_c$,
\begin{equation}
\eta_k^e = \big( \eps \sum_{j = 0}^2  [[ \phi'
]]_{\ell_{La_i}+j}^2 + \frac{1}{2}\eps \sum_{j=0}^2[[ \phi'
]]_{\ell_{k-1}+j}^2 \big)^{\frac{1}{2}},
\end{equation}
if $k = L_{a_i}$ for some $i \in \{1,2,\ldots,M\}$, i.e., $T_k$ is adjacent to and to the left of an atomistic region, and
\begin{equation}
\eta_k^e = \big( \eps \sum_{j = 0}^2  [[ \phi'
]]_{\ell_{Ra_i}+j}^2 + \frac{1}{2}\eps \sum_{j=0}^2[[ \phi'
]]_{\ell_{k}+j}^2 \big)^{\frac{1}{2}},
\end{equation}
if $k-1 = R_{a_i}$ for some $i \in \{1,2,\ldots,M\}$, i.e., $T_k$ is adjacent to and to the right of an atomistic region. $[[\phi']]_\ell$'s will be defined in the proof.
\end{theorem}
\begin{proof}
By \eqref{VariationalAStore} and \eqref{VariationalQCStore}, we have
\begin{align}
\label{ResidualStore1}
\Es'_\a(J_{\Us_{\qc}} y_h)[v]-\Es'_{\qc}(y_h)[J_\Us v]
=&\eps \sum_{b
  \in \Bs} \phi'(r_b D_b J_{\Us_{\qc}} y_h) r_bD_b v-\sum_{b \in \Bs}|b\cap \Omega_\a| \phi'(r_b D_{b \cap \Omega_\a} y_h) D_{b \cap
  \Omega_\a} J_\Us v  \nonumber\\
 &-\sum_{b \in \Bs} \frac{1}{r_b} \int_{b \cap
  \Omega_c} \phi'(\nabla_{r_b} y_h) \nabla_{r_b} J_\Us v \dx\nonumber\\
=&\eps \sum_{b \in \Bs} \phi'(r_b D_b y_h) r_bD_b v -\sum_{b \in \Bs}|b\cap \Omega_\a| \phi'(r_b D_{b \cap \Omega_\a} y_h) D_{b \cap
  \Omega_\a} I_hv  \nonumber\\
  &- \sum_{b \in \Bs} \frac{1}{r_b} \int_{b \cap
  \Omega_c} \phi'(\nabla_{r_b} y_h) \nabla_{r_b} I_hv \dx,
\end{align}
since $D_b J_{\Us_{\qc}} y_h = D_b I_\eps y_h = D_b y_h$, $D_\omega J_{\Us} v =
D_\omega I_h v$ for any $\omega$ being an open interval, and $(J_\Us v)' = (I_hv)'$, which can be easily verified by
noting that $J_{\Us_{\qc}} y_h$ and $J_\Us v$ are $I_\eps y_h$ and
$I_h v$ shifted by some constants.

To make further analysis of \eqref{ResidualStore1}, we subtract and
add the same terms 
\begin{displaymath}
\sum_{b \in \Bs}|b\cap \Omega_\a| \phi'(r_b D_{b \cap \Omega_\a} y_h)  D_{b \cap
  \Omega_\a}v \text{ and } \sum_{b \in \Bs} \frac{1}{r_b} \int_{b \cap
  \Omega_c} \phi'(\nabla_{r_b} y_h) \nabla_{r_b} v \dx
\end{displaymath}
to get
\begin{align}
\label{ResidualStore2}
\Es'_\a(J_{\Us_{\qc}} y_h)[v]-\Es'_{\qc}(y_h)[J_\Us v]
 = &\sum_{b
  \in \Bs} \bigg\{   \eps \phi'(r_b D_b  y_h) r_bD_b v -|b\cap \Omega_\a| \phi'(r_b D_{b \cap \Omega_\a} y_h)  D_{b \cap
  \Omega_\a} v \nonumber\\
   & -\frac{1}{r_b} \int_{b \cap
  \Omega_c} \phi'(\nabla_{r_b} y_h) \nabla_{r_b} v \dx \bigg\} \nonumber\\
   &-\bigg\{ \sum_{b \in \Bs}
|b\cap \Omega_\a| \phi'(r_b D_{b \cap \Omega_\a} y_h) \big[ D_{b \cap
  \Omega_\a} I_hv \dx - D_{b \cap \Omega_\a} v \big] \bigg\}\nonumber\\
   &-\bigg\{\sum_{b \in \Bs} \frac{1}{r_b} \int_{b \cap
  \Omega_c} \phi'(\nabla_{r_b} y_h) \big[\nabla_{r_b} I_hv -
\nabla_{r_b} v \big] \dx \bigg\}.
\end{align}

We first analyze the second and third groups, which turn out to be $0$
as we will see immediately.

For the second group, we have,
\begin{align}
\label{ResidualStore2-2}
&\quad \sum_{b \in \Bs}
|b\cap \Omega_\a| \phi'(r_b D_{b \cap \Omega_\a} y_h) \big[ D_{b \cap
  \Omega_\a} I_hv \dx - D_{b \cap \Omega_\a} v \big] \nonumber \\
&= \sum_{b \in \Bs}
|b\cap \Omega_\a| \phi'(r_b D_{b \cap \Omega_\a} y_h) \bigg[
\frac{I_hv(R_{b \cap \Omega_\a}) - I_hv(L_{b \cap
    \Omega_\a})}{R_{b \cap \Omega_\a}-L_{b \cap \Omega_\a}} - \frac{v(R_{b \cap \Omega_\a}) - v(L_{b \cap
    \Omega_\a})}{R_{b \cap \Omega_\a}-L_{b \cap \Omega_\a}} \bigg].
\end{align}
We define the above to be $0$ if $b \cap \Omega_\a = \emptyset$. If $b \cap
\Omega_\a \neq \emptyset$, since both $R_{b \cap
  \Omega_\a}$ and $L_{b \cap \Omega_c}$ are either at atomistic postions
in $\Omega_\a$ or on $\partial \Omega_c$, they must be nodes
in $\Ts^h$. Therefore, by the definition of $I_hv$, the following holds 
\begin{displaymath}
I_hv(L_{b \cap
    \Omega_\a}) = v(L_{b \cap \Omega_\a}) \text{ and } I_hv(R_{b \cap \Omega_\a}) = v(R_{b \cap \Omega_\a}),
\end{displaymath}
which implies that \eqref{ResidualStore2-2} is $0$.

For the third group, upon defining $\chi_{\mathcal{S}}$ to be the characteristic function of a set
$\mathcal{S}$, we can rewrite it as
\begin{align}
\label{ResidualStore2-3-1}
& \quad \sum_{b \in \Bs} \frac{1}{r_b} \int_{b \cap
  \Omega_c} \phi'(\nabla_{r_b} y_h) \big[\nabla_{r_b} I_hv -
\nabla_{r_b} v \big] \dx \nonumber\\ 
&= \sum_{b \in \Bs} \frac{1}{r_b} \int_{\Omega_c} \chi_b \phi'(\nabla_{r_b} y_h) \big[\nabla_{r_b} I_hv -
\nabla_{r_b} v \big] \dx \nonumber\\
&=\sum_{b \in \Bs} \frac{1}{r_b} \sum_{k\in \Ks_c}\int_{T_k} \chi_b \phi'(\nabla_{r_b} y_h) \big[\nabla_{r_b} I_hv -
\nabla_{r_b} v \big] \dx \nonumber\\
&= \sum_{r = 1}^2 \sum_{b \in \Bs, r_b=r} \sum_{k \in \Ks_c} \frac{1}{r_b} \int_{T_k} \chi_b \phi'(\nabla_{r_b} y_h) \big[\nabla_{r_b} I_hv -
\nabla_{r_b} v \big] \dx \nonumber\\
&=\sum_{r = 1}^2 \sum_{k \in \Ks_c}\sum_{b \in \Bs, r_b=r} \frac{1}{r_b} \int_{T_k} \chi_b \phi'(\nabla_{r} y_h) \big[\nabla_{r} I_hv -
\nabla_{r} v \big] \dx \nonumber\\
&= \sum_{r = 1}^2 \sum_{k \in \Ks_c} \phi'(\nabla_{r} y_h |_{T_k}) \int_{T_k} \big[ \sum_{b \in \Bs, r_b=r} \frac{1}{r_b} \chi_b \big] \big[\nabla_{r} I_hv -
\nabla_{r} v \big],
\end{align}
since $\nabla_{r} y_h |_{T_k}$ is a constant on each element. By the 1D bond density lemma\cite[Lemma 3.4]{Shapeev2010a}, 
\begin{displaymath}
\sum_{b \in \Bs, r_b=r} \frac{1}{r}\chi_b(x) =_{a.e.}1,
\end{displaymath}
we have
\begin{align}
&\sum_{r = 1}^2 \sum_{k \in \Ks_c} \phi'(\nabla_{r} y_h |_{T_k}) \int_{T_k} \big[ \sum_{b \in \Bs, r_b=r} \frac{1}{r_b} \chi_b \big] \big[\nabla_{r} I_hv -
\nabla_{r} v \big] \nonumber \\
=&\sum_{r = 1}^2 \sum_{k \in \Ks_c} \phi'(\nabla_{r_b} y_h|_{T_k})\bigg[ r
\big(I_hv(x_k) -I_hv(x_{k-1})\big) - r \big(v(x_k) -v(x_{k-1})\big)
\bigg].
\label{ResidualStore2-3-2}
\end{align}
Again by the definition of $I_hv$, 
\begin{displaymath}
I_hv(x^h_k) = v(x^h_k) \text{ and } I_hv(x^h_{k-1})=v(x^h_{k-1}),
\end{displaymath}
and thus \eqref{ResidualStore2-3-2} is $0$.

Now we turn to the analysis of the first group and analyze 
\begin{equation}
\eps \phi'(r_b D_b  y_h) r_bD_b v -|b\cap \Omega_\a| \phi'(r_b D_{b \cap \Omega_\a} y_h)  D_{b \cap
  \Omega_\a} v 
    -\frac{1}{r_b} \int_{b \cap
  \Omega_c} \phi'(\nabla_{r_b} y_h) \nabla_{r_b} v \dx
\label{ResidualStore2-1}
\end{equation}
for each interaction bond $b$.

If $b \subset \Omega_\a$, we have $|b \cap \Omega_c| = r_b \eps$, $|b \cap \Omega_\a| = 0$ 
and the equivalence of the operators $D_b = D_{b \cap \Omega_\a}$. We know that \eqref{ResidualStore2-1} is $0$ by substituting these equivalences.

If $b \subset \Omega_c \cap
T_k$ for some $k \in \Ks_c$, then $|b \cap \Omega_c| = r_b \eps$ and $|b \cap \Omega_\a| = 0$. We also note that $\nabla_{r_b} y_h (x)= r_bD_by_h$, 
as $y_h$ is affine on $T_k$, and $\frac{1}{r_b}\int_{b \cap \Omega_c} \nabla_{r_b} v = \eps r_bD_bv$. Using these equivalences, we know that \eqref{ResidualStore2-1} is again $0$.

Therefore, we only need to analyze the bonds crossing the atomistic-continuum interface 
or the boundaries of two adjacent elements in $\Omega_c$. Because of its tediousness, we leave the detailed analysis to the Appendix but just present the result here. 
Employing the notation often adopted by a posteriori error analysis for elliptic equations, we have the
following result
\begin{align}
\label{ResiStorLastPar}
&\quad \eps \sum_{b
  \in \Bs} \phi'(r_b D_b I_\eps y_h) r_bD_b v -\sum_{b \in \Bs}
|b\cap \Omega_\a| \phi'(r_b D_{b \cap \Omega_\a} y_h)  D_{b \cap
  \Omega_\a} v + \sum_{b \in \Bs} \frac{1}{r_b} \int_{b \cap
  \Omega_c} \phi'(\nabla_{r_b} y_h) \nabla_{r_b} v \dx  \nonumber \\
&=\sum_{i=1}^M \eps \bigg\{ [[\phi']]_{\ell_{La_i}} v'_{\ell_{La_i}}
  +[[ \phi' ]]_{\ell_{La_i}+1} v'_{\ell_{La_i}+1} +[[ \phi' ]]_{\ell_{La_i}+2} v'_{\ell_{La_i}+2} \bigg\}
  \nonumber\\
&\quad+\sum_{i=1}^M \eps \bigg\{ [[\phi']]_{\ell_{Ra_i}} v'_{\ell_{Ra_i}}
  +[[ \phi' ]]_{\ell_{Ra_i}+1} v'_{\ell_{Ra_i}+1} +[[ \phi' ]]_{\ell_{Ra_i}+2} v'_{\ell_{Ra_i}+2} \bigg\} \nonumber\\
&\quad+\sum_{k\in \Ks'_c} \eps  \bigg\{ [[\phi']]_{\ell_{k}} v'_{\ell_{k}}
  +[[ \phi' ]]_{\ell_{k}+1} v'_{\ell_{k}+1} +[[ \phi' ]]_{\ell_{k}+2} v'_{\ell_{k}+2} \bigg\},
\end{align}
where for $k =La_i$,
\begin{align}
[[\phi']]_\ell &= \phi'\big(
(1-\theta_k) y_h'|_{T_{k+1}} 
+(1+\theta_k) y_h'|_{T_k} \big) - \phi'\big(2y_h'|_{T_k}\big),
\end{align}
\begin{align}
[[\phi']]_{\ell+1} &= \bigg[\phi'\big( (1-\theta_k) y_h'|_{T_{k+1}} +
\theta_k y_h'|_{T_k} \big) - (1-\theta_k)\phi'\big(y_h'|_{T_{k+1}}\big) -
\theta_k \phi'\big(y_h'|_{T_k}\big) \bigg]  \nonumber\\
&\quad+\bigg[ (1-\theta_k)
\phi'\big(\frac{2}{2-\theta_k}y_h'|_{T_{k+2}} + \frac{2(1-\theta_k)}{2-\theta_k}y_h'|_{T_k+1} \big) \nonumber\\
&\quad \quad +\theta_k\phi'\big(2y_h'|_{T_k}\big)  -\phi'\big(
y_h'|_{T_{k+2}} + (1-\theta_k)y_h'|_{T_{k+1}}+\theta_k y_h'|_{T_{k}}\big) \bigg] \nonumber\\
&\quad+\bigg[ (1-\theta_k) \phi'\big(2y_h'|_{T_{k+1}}\big) + \theta_k \phi'\big(2y_h'|_{T_k}\big) -\phi'\big((1-\theta_k)y_h'|_{T_{k+1}} +(1+\theta_k) y_h'|_{T_k}\big) \bigg]
\end{align}
\begin{align}
[[\phi']]_{\ell+2} &=
\phi'(\frac{2(1-\theta_k)}{2-\theta_k}y_h'|_{T_{k+1}}+\frac{2}{2-\theta_k}
y_h'|_{T_{k+2}}\big)
-\phi'\big(y_h'|_{T_{k+2}} + (1-\theta_k) y_h'|_{T_{k+1}}
+\theta_ky_h'|_{T_k}\big),
\end{align}
for $k =Ra_i$
\begin{align}
[[\phi']]_{\ell} &= \phi'\big((1-\theta_k) y_h'|_{T_{k+1}} + \theta_k y_h'|_{T_k} +y_h'|_{T_{k-1}}\big) - 
\phi'\bigg(\frac{2\theta_k}{(1+\theta_k)}y_h'|_{T_k}+\frac{2}{1+\theta_k}
y_h'|_{T_{k-1}}\bigg), 
\end{align}
\begin{align}
[[\phi']]_{\ell+1} &= \bigg[\theta_k\phi'\big(y_h'|_{T_k}\big) +
(1-\theta_k)\phi'\big(y_h'|_{T_{k+1}}\big) -\phi'\big((1-\theta_k) y_h'|_{T_{k+1}} + \theta_k y_h'|_{T_k}  \big) \bigg]  \nonumber\\
&\quad+\bigg[ \theta_k \phi'\big(\frac{2\theta_k}{1+\theta_k} y_h'|_{T_k}
+\frac{2}{1+\theta_k}y_h'|_{T_{k-1}}\big) +
(1-\theta_k)\phi'\big(2y_h'|_{T_{k+1}}\big)\nonumber\\
& \quad-\phi'\big((1-\theta_k) y_h'|_{T_{k+1}}+\theta_k y_h'|_{T_k} +y_h'|_{T_{k-1}}\big) \bigg] \nonumber\\
&\quad +\bigg[ (1-\theta_k)
\phi'\big(2y_h'|_{T_{k+1}}\big)+\theta_k\phi'\big(2y_h'|_{T_k}\big) -\phi'\big((2-\theta_k) y_h'|_{T_{k+1}}+\theta_ky_h'|_{T_k} \big) \bigg]
\end{align}
\begin{align}
[[\phi']]_{\ell+2} &= \phi'\big(2y_h'|_{T_{k+1}} \big) - \phi'\big((2-\theta_k) y_h'|_{T_{k+1}}+\theta_k y_h'|_{T_k} \big),
\end{align}
and for $k \in \Ks'_c$
\begin{align}
[[\phi']]_{\ell} &= \phi'\big(2y_h'|_{T_k}\big) -
\phi'\big( (1-\theta_k) y_h'|_{T_{k+1}} +
(1+\theta_k) y'_h|_{T_k}\big),
\end{align}
\begin{align}
[[\phi']]_{\ell+1} &= \bigg[(1-\theta_k)\phi'\big(y'_h|_{T_{k+1}}\big)+\theta_k \phi'\big(y'_h|_{T_k}\big)-\phi'\big((1-\theta_k) y_h'|_{T_{k+1}}+\theta_k y_h'|_{T_k}\big)\bigg] \nonumber\\
\quad&+\bigg[2(1-\theta_k)\phi'\big(2y_h'|_{T_{k+1}}\big)+2\theta_k\phi'\big(2y_h'|_{T_k}\big)
\nonumber\\
\quad&-\phi'\big( (2-\theta_k) y_h'|_{T_{k+1}}+\theta_k y_h'|_{T_k} \big)-\phi'\big( (1-\theta_k) y_h'|_{T_{k+1}}+(1+\theta_k) y_h'|_{T_k}\big)
\bigg],
\end{align}
\begin{align}
[[\phi']]_{\ell+2}&=\phi'\big(
2y_h'|_{T_{k+1}}\big)-\phi'\big((2-\theta_k)y_h'|_{T_{k+1}}+\theta_k y_h'|_{T_k}\big).
\end{align}
Distributing the contribution of \eqref{ResiStorLastPar} to each element
and applying Cauchy-Schwarz inequality, we obtain the estimate stated in the theorem.
\end{proof}

\subsection{Estimate of the Residual of the External Force}
\label{SectionResidualExternalForce}
We now turn to the estimate of the residual of the
external energy. Upon defining $J_\Us v := v_h$, the residual of the
external force is given by 
\begin{equation}
\<f, v_h\>_h - \<f,v\>_\eps,
\label{Eq:ResGraExEnOrig}
\end{equation}
where $f, v \in \Us$. 

To further analyze \eqref{Eq:ResGraExEnOrig}, we introduce a new
partition $\Ts^r = \{T^r_j\}_{j=-\infty}^{+\infty}$ of the domain
$\R$, such that all the nodes in partition $\Ts^\eps$ and partition
$\Ts^h$ are included in this partition. The indexing of the nodes in
$\Ts^r$ follow the rule that the node $x^h_k$ in $\Ts^h$ is labeled as $x_{j_k}^r$ in
  $\Ts^r$. We also assume there are $n$ nodes in $\Ts^r$ in $[0,1]$, i.e.,  
\begin{displaymath}
n = \big|\{\eps,2\eps,\ldots,N \eps\} \cup \{ x_1,x_2,\ldots,x_K \}\big|,
\end{displaymath} 
where $|\As|$ denote the cardinality of a finite set $\As$.

The inner product associated with $\Ts^r$ partition is then defined by
\begin{equation}
\label{rInnerproduct}
\<f,g\>_r :=  \int_0^1 I_r (f g) \dx =\sum_{j = 1}^n \frac{1}{2} (x^r_{j+1}-x^r_{j-1})
f^r_jg^r_j = : \<\fb^r, \gb^r\>_r \quad \forall f,g \in C^0(\R),
\end{equation}
where $I_r$ is the linear nodal interpolation operator with respect to
$\Ts^r$, and $\fb^r$ and $\gb^r$ are the vectorizations of $f$ and $g$
with respect to $\Ts^r$.

Now we decompose the residual of the external force into three parts by adding and
subtracting the same terms, 
\begin{align}
\<f,v_h\>_h-\<f,v\>_\eps   = \big[\<f,v\>_r- \<f,v\>_\eps\big]
                                              +\big[\<f,
                                             v_h\>_r-\<f,v\>_r\big] 
                                               +\big[ \<f,v_h\>_h -\<f,v_h\>_r \big]
\label{Eq:ResGraExEnSplit}.
\end{align}

The following three lemma are derived to give the estimates of the three
parts.

\begin{lemma}
\label{Lemma:ResGraExEn-1}
Let $\fb, \vb, \fb^r, \vb^r$ be the vectorizations of $f, v \in C^0(\R)$ according
to $\Ts^\eps$ and $\Ts^r$. Then the following inequality holds
\begin{equation}
\big|\<f,v\>_r- \<f,v\>_\eps\big| = \big|\<\fb^r,
\vb^r\>_r-\<\fb,\vb\>_\eps \big| \le  \frac{1}{8} \eps^2 \|\fb'\|_{\ell_\eps^2 (\mathcal{K}_U)} \|\vb'\|_{\ell_\eps^2},
\end{equation}
where $\mathcal{K}_U = \big\{k \in\{ 1, \ldots, K\}: x_k \neq \ell_k\eps
\big\}$, in other words, $\Ks_U$ is the set of indices of the nodes
$x^h_k$ in $\Ts^h$ such that $x^h_k$ does not coincide with any of the nodes in $\Ts^\eps$.
\end{lemma}
\begin{proof}
We first write out the two inner products and eliminate the terms that
are the same
\begin{align}
\<\fb^r, \vb^r\>_r-\<\fb,\vb\>_\eps
&= \sum_{\ell = 1}^n \eps^r_\ell \frac{1}{2}
(f^r_\ell v^r_\ell + f^r_{\ell+1} v^r_{\ell+1}) -\sum_{\ell = 1}^N \eps \frac{1}{2}(f_\ell v_\ell+f_{\ell+1}
v_{\ell+1}) \nonumber \\
&= \sum_{k \in \Ks_U}\big( \eps_{j_k}\frac{1}{2} ( f^r_{j_k-1} v^r_{j_k-1} + f^r_{j_k}
v^r_{j_{k}}) +  \eps_{j_k+1}\frac{1}{2} ( f^r_{j_{k}} v^r_{j_{k}} + f^r_{j_{k}+1}
v^r_{j_{k}+1}) \big)\nonumber\\
&\quad -\sum_{k\in \Ks_U} \eps \big( f_{\ell_k} v_{\ell_k} + f_{\ell_{k}+1}
v_{\ell_{k}+1}\big) ,
\label{Eq:ResGraExEnSplit-1-1}
\end{align}
as $ \eps_{j_k+2} = \eps_{j_k+3} = \ldots =
\eps_{j_{k+1}-1} = \eps$ and $ f_{\ell_k+i} v_{\ell_k+i}  = f^r_{j_{k}+i}
v^r_{j_{k}+i}, \quad i = 1, 2,\ldots, \ell_{k+1} -\ell_{k}$, if
$\ell_k\eps \neq x_k$ and $\ell_{k+1}\eps \neq x_{k+1}$.

For $k$ such that $\ell_k\eps \neq x_k$, by the definition of $\fb$,
$\vb$, $\fb^r$ and $\vb^r$, we have $f_{\ell_k} =f^r_{j_k-1}$,
$v_{\ell_k} =v^r_{j_k-1}$, $f_{\ell_k+1} =f^r_{j_k+1}$ and
$v_{\ell_k+1} =v^r_{j_k+1}$.  We also have $f^r_{j_{k}} =   (1-\theta_k) f_{\ell_k} +\theta_k
f_{\ell_{k}+1}$ and $v^r_{j_{k}} =   (1-\theta_k) v_{\ell_k} +\theta_k
v_{\ell_{k}+1}$. Inserting these equalities, \eqref{Eq:ResGraExEnSplit-1-1} can
be estimated as

\begin{align}
\label{ErrExEn4-1-2}
&\big|\<\fb^r,\vb^r\>_r - \<\fb,\vb\>_\eps \big|\nonumber\\
=&\bigg| \sum_{k \in \Ks_U}
\bigg\{\frac{1}{2} \theta_k\eps  f_{\ell_k}v_{\ell_k} + \frac{1}{2}
\theta_k\eps \big[  (1-\theta_k) f_{\ell_k} +\theta_k
f_{\ell_{k+1}}\big] \big[ (1-\theta_k) v_{\ell_k} +\theta_k
v_{\ell_{k}+1}\big] \nonumber\\
&+ \frac{1}{2} (1-\theta_k)\eps  f_{\ell_k+1}v_{\ell_k+1} + \frac{1}{2}
(1-\theta_k)\eps \big[  (1-\theta_k) f_{\ell_k} +\theta_k
f_{\ell_{k}+1}\big] \big[ (1-\theta_k) v_{\ell_k} +\theta_k
v_{\ell_{k}+1}\big] \nonumber\\
&-\frac{1}{2} \eps  f_{\ell_k}v_{\ell_k} -\frac{1}{2} \eps
f_{\ell_k+1}v_{\ell_k+1} \bigg\} \bigg|\nonumber\\
= &\big|\sum_{k \in \Ks_U} \frac{1}{2} \eps \bigg\{ \big[ \theta_k(\theta_k-1)
(f_{\ell_k+1}- f_{\ell_k}) v_{\ell_k+1}\big] -\big[ \theta_k(\theta_k-1)
(f_{\ell_k+1}- f_{\ell_k}) v_{\ell_k}\big] \bigg\} \big|\nonumber\\
= &\sum_{k \in \Ks_U}\frac{1}{2} \eps^3 \big|\theta_k (1-\theta_k)  f'_{\ell_k+1}
v'_{\ell_k+1} \big|\nonumber \\
&\le \frac{1}{8} \eps^2 \bigg( \sum_{k \in \Ks_U} \eps
|f'_{\ell_k+1}|^2\bigg)^{\frac{1}{2}} \bigg( \sum_{k \in \Ks_U} \eps
|v'_{\ell_k+1}|^2\bigg)^{\frac{1}{2}}  \le \frac{1}{8} \eps^2 \|\fb'\|_{\ell_\eps^2 (\mathcal{K}_U)} \|\vb'\|_{\ell_\eps^2},
\end{align}
which concludes the proof.
\end{proof}

\begin{remark}
If $\Ks = \emptyset$, i.e., every node in $T^h$ is also in $T^\eps$, then this part of the residual is $0$.
\end{remark}

\begin{lemma}
\label{Lemma:ResGraExEn-2}
Let $f,v \in C^0(\R) \cap \Ps_1(\Ts^\eps)$ and $v_h = I_h v \in C^0(\R) \cap \Ps_1(\Ts^h)$ be
the $\Ps_1$ interpolation of $v$ according to $\Ts^h$
partition. Let $\fb^r, \vb^r$ and $\vb_h^r$ be the vectorizations of
$f, v, v_h$ respectively according
to $\Ts^r$, and $\Ks_c$ is defined in \eqref{Set:K_c}. Then we have the following estimate
\begin{equation}
\<f, v_h\>_r - \<f,v\>_r = \<\fb^r, \vb_h^r\>_r - \<\fb^r, \vb^r\>_r
\le \bigg[ \sum_{k=\Ks_c} \tilde{h}_k^2
\|\fb^r\|^2_{\ell^2_{\bar{\eps}^r} (\Ds^2_k)}\bigg]^{\frac{1}{2}} \| \vb'\|_{\ell^2_\eps},
\end{equation}
$\bar{\eps}^r_j = \frac{1}{2}(\eps^r_j+\eps^r_{j+1})$,
$\tilde{h}_k = \frac{1}{2}(j_{k+1}-j_k)\eps$ and $\Ds^2_k = \{j_k+1,
\ldots, j_{k+1}-1\}$. 
\end{lemma}

\begin{proof}
Using the fact that $(v_h^r)_{j_k} = v^r_{j_k}$ and by Cauchy-Schwarz
inequality, we have
\begin{align}
\<\fb^r, \vb_h^r\>_r - \<\fb^r, \vb^r\>_r = &\sum_{j = 1}^n \frac{1}{2} (\eps^r_j+\eps^r_{j+1}) (f^r_j
(v_h^r)_j-f^r_j v^r_j) \nonumber\\
= &\sum_{k \in \Ks_c} \sum_{j = j_{k-1}} ^{j_{k}-1} \bar{\eps}^r_j f^r_j
\big[ (v_h^r)_j- v^r_j \big] \nonumber\\
\le&\sum_{k \in \Ks_c}\bigg[ \sum_{j = j_{k-1}+1} ^{j_{k}-1} \bar{\eps}^r_j
(f^r_j)^2 \bigg]^{\frac{1}{2}} \bigg\{\sum_{j = j_{k-1}+1} ^{j_{k}-1}
\bar{\eps}^r_j \big[ (v_h^r)_j- v^r_j \big]^2\bigg\}^{\frac{1}{2}},
\end{align}
where $\bar{\eps}^r_j = \frac{1}{2}(\eps^r_j+\eps^r_{j+1})$. 
Upon defining $\gb$ such that $g_j = (v_h^r)_j- v^r_j$ (note $g_{j_k}
= g_{j_{k+1}}= 0$) and by Lemma \ref{Lemma:DiscFriedrichNonUniMesh} in Appendix
\ref{AppendixB} (Discrete Friedrich's Inequality) and Rieze-Thorin Theorem,
\begin{align}
\bigg\{\sum_{j = j_{k-1}+1} ^{j_{k}-1}
\bar{\eps}^r_j \big[ (v_h^r)_j- v^r_j \big]^2\bigg\}^{\frac{1}{2}} = &\bigg\{\sum_{j = j_{k-1}} ^{j_{k}}
\bar{\eps}^r_j g_j^2\bigg\}^{\frac{1}{2}}
\le \frac{1}{2}(j_{k}-j_{k-1})\eps \bigg\{ \sum_{j=j_{k-1}+1}^{j_{k}}
\eps^r_j {g'_j}^2 \bigg\}^{\frac{1}{2}},
\end{align}
where $g_j' = \smfrac{g_j-g_{j-1}}{\eps^r_j } =
(v^r)'_j-(v_h^r)'_j$. $\eps$ appears in the
last inequality since $\max_j \bar{\eps}^r_j \le \eps$. Since $v_h^r$
and $v^r$ are both piecewise linear on $\Ts^r$, we have $(v_h^r)'_j - (v^r)'_j = (v'-v_h')(x) \ \forall x \in
(x^r_{j-1}, x^r_j)$, and as a result, 
\begin{equation*}
 \bigg\{ \sum_{j=j_{k-1}+1}^{j_{k}}
\eps^r_j \big[ (v_h^r)'_j - (v^r)'_j\big]^2 \bigg\}^{\frac{1}{2}} = \int_{x^r_{j_{k-1}}}^ {x^r_{j_{k}}}  |(v'-v_h')(x)|^2
\dx = \| v'-v_h'\|_{L^2[x^r_{j_{k-1}}, x^r_{j_{k}}]}.
\end{equation*}
By Lemma \ref{lemmaA2} in Appendix \ref{AppendixA}, 
\begin{equation*}
\| v'-v_h'\|_{L^2[x^r_{j_{k-1}}, x^r_{j_{k}}]} \le \| v'\|_{L^2[x^r_{j_{k-1}}, x^r_{j_{k}}]}.
\end{equation*}
Put all the results above together and apply Cauchy-Schwarz inequality, we obtain 
\begin{align}
&\sum_{k \in \Ks_c}\bigg[ \sum_{j = j_{k-1}+1} ^{j_{k}-1} \bar{\eps}^r_j
(f^r_j)^2 \bigg]^{\frac{1}{2}} \bigg\{\sum_{j = j_{k-1}+1} ^{j_{k}-1}
\bar{\eps}^r_j \big[ (v_h^r)_j- v^r_j \big]^2\bigg\}^{\frac{1}{2}}
\nonumber \\
\le &\sum_{k \in \Ks_c} \Bigg\{ \tilde{h}_k \Big(\sum_{j =
 j_{k-1}+1 }^{j_{k}-1} \bar{\eps}^r_j
{f^r_j}^2\Big)^{\frac{1}{2}} \| v'\|_{L^2(x^r_{j_{k-1}}, x^r_{j_{k}})} \Bigg\}
 \le \bigg[ \sum_{k=\Ks_c} \tilde{h}_k^2 \Big(\sum_{j =
 j_{k-1}+1 }^{j_{k}-1} \bar{\eps}^r_j
{f^r_j}^2\Big) \bigg]^{\frac{1}{2}} \| v'\|_{L^2{[0,1]}}.
\end{align}
The eatimate in the theorem holds as $\| v'\|_{L^2{[0,1]}} = \|
\vb'\|_{\ell^2_\eps}$ for $v \in C^0(\R) \cap \Ps_1(\Ts^\eps)$.
\end{proof}

\begin{lemma}
\label{Lemma:ResGraExEn-3}
Let $f,v \in C^0(\R) \cap \Ps_1(\Ts^\eps)$ and $v_h = I_h v \in C^0(\R) \cap \Ps_1(\Ts^h)$ be
the $\Ps_1$ interpolation of $v$ according to the $\Ts^h$. Let
$\fb^r,\vb^r$ and $\vb_h^r$ be the vectorizations $f, v$ and $v_h$ according
to $\Ts^r$. If $\int_0^1 v_h = 0$, then we have the following estimate
\begin{equation}
\<f, v_h\>_h - \<f,v_h\>_r \le \Bigg\{ \frac{1}{8}\bigg[ (n\eps)^4\sum_{k= \Ks_c}\hat{h}_{k+1}^4
\|{\fb^r}''\|^2_{\ell^2_{\bar{\eps}^r}(\Ds^2_k)}\bigg]^{\frac{1}{2}}
+\bigg[ \sum_{k=\Ks_c} \hat{h}_{k+1}^4
\|{\fb^r}'\|^2_{\ell^2_{\eps^r(\Ds^1_k)}}\bigg]^{\frac{1}{2}} \Bigg\}\|\vb'\|_{\ell^2_\eps}.
\end{equation}
where $\Ks_c$ is defined in \eqref{Set:K_c}, $\Ds_k^1
= \{j_k+1, \ldots, j_{k+1}\} $ and $\hat{h}_k$ will be defined in the proof.
\end{lemma}
\begin{proof}
Since $I_h(f v_h)$ is also piecewise linear with respect to the $\Ts^r$ partition,
we apply the trapezoidal rule here to evaluate $\<f,v_h\>_h =
\int_0^1 I_h(fv_h) \dx$ to obtain
\begin{align}
&\<f,v_h\>_h - \<f,v_h\>_r \nonumber\\
=&\sum_{k \in \Ks_c} \Bigg\{ \bigg[ \frac{1}{2}\eps^r_{j_{k-1}+1}
I_h(fv_h)(x^r_k)+\sum_{j = j_{k-1}+1}^{j_{k}-1}
\frac{1}{2}(\eps^r_{j}+\eps^r_{j+1}) I_h (f
v_h)(x^r_{j})+\frac{1}{2}\eps^r_{j_{k}}I_h(fv_h)(x^r_{j_{k}})
\bigg]  \nonumber\\
& -\bigg[ \frac{1}{2}\eps^r_{j_{k-1}+1}
(fv_h)(x^r_{j_{k-1}+1})+ \sum_{j = j_{k-1}+1}^{j_{k}-1}
\frac{1}{2}(\eps^r_{j}+\eps^r_{j+1})
(fv_h)(x^r_{j})+\frac{1}{2}\eps^r_{j_{k}}(fv_h)(x^r_{j_{k}})\bigg]
\Bigg\}.
\end{align}

We define $\gb$ and $\Gb$ such that $g_j = (fv_h)(x^r_j)$ and $G_j = (I_h(f
v_h))(x^r_j)$. It is easy to check that $g_{j_k} = G_{j_k}$ and 
\begin{equation}
G_{j_{k-1}+i} = g_{j_{k-1}} +
\frac{\sum_{\ell=1}^i\eps^r_{j_{k-1}+\ell}}{\eps^h_{k}}
(g_{j_{k}}-g_{j_{k-1}})\quad \forall k \in \Ks_c \quad and
\quad i = 1, \ldots, j_{k}-j_{k-1},
\end{equation}
where $\eps^h_{k} =\sum_{j=j_{k-1}
  +1}^{j_{k}}\eps^r_{j} =x_{k}-x_{k-1}$. Therefore, by Theorem \ref{Theo:BoundOnIntErr}, we obatin the following estimate
\begin{align}
&|\<f,v_h\>_h - \<\fb^r,I_r\vb^h\>_r | \nonumber\\
= &\Bigg|\sum_{k \in \Ks_c}\bigg[ \frac{1}{2}\eps^r_{j_{k-1}}
(g_{j_{k-1}}-G_{j_{k-1}}) +\sum_{j = j_{k-1}+1}^{j_{k}-1}
\frac{1}{2}(\eps^r_{j}+\eps^r_{j+1}) (g_{j}-G_{j})+\frac{1}{2}\eps^r_{j_{k}}(g_{j_{k}}-G_{j_{k}})\bigg] \Bigg|
\nonumber\\
\le& \sum_{k \in \Ks_c}\frac{1}{4}
\frac{\big((j_{k}-j_{k-1})\eps\big) \big(
  (j_{k}-j_{k-1}+1)\eps\big)^2}{\eps^h_{k}} \|\gb''\|_{\ell_{\bar{\epsb}^r}^1 (\mathcal{D}^2_k)},
\end{align}
where $\gb''$ is the second finite difference derivative with respect
to the $\Ts^r$. 

By the definition of $\fb^r$ and $\vb_h^r$, $g_j =(fv_h)(x^r_j)
=f^r_{j} (v_h^r)_{j}$. Using $(v_h^r)''_j = 0 \  \forall j \in
\Ds^2_k$, $g_j''$ can be written as
\begin{equation}
g_j''=(fv_h^r)''_j = (f^r)''_j(v_h^r)_j +
\frac{\eps^r_{j}}{\bar{\eps}^r_j}(f^r)'_{j}(v_h^r)'_{j} +\frac{\eps^r_{j+1}}{\bar{\eps}^r_j}(f^r)'_{j+1}(v_h^r)'_{j+1}.
\end{equation}

Noting that $\frac{\eps^r}{\bar{\eps}_j} \le 2$ and
  $\frac{\eps^r_{j+1}}{\bar{\eps}_j} \le 2$ and defining $\hat{h}_{k} :=
\bigg[\frac{((j_{k}-j_{k-1})\eps) 
  (j_{k}-j_{k-1}+1)\eps\big)^2}{\eps^h_{k}}\bigg]^{\frac{1}{2}}$, we have the following estimate
\begin{align}
\<f,v_h\>_h - \<f,v_h\>_r \le &\sum_{k\in \Ks_c}\frac{1}{4}\hat{h}_k^2 \bigg[ \sum_{j =
  j_{k-1}+1}^{j_{k}-1} \bar{\eps}^r_j \big|(fv_h^r)''_j\big|\bigg]
\nonumber\\
\le & \sum_{k \in \Ks_c} \frac{1}{4}\hat{h}_{k}^2 \bigg[ \sum_{j
  =j_{k-1}+1}^{j_{k} -1}  \bar{\eps}^r_j \big|(f^r)''_j\big| \big| (v_h^r)_j \big| +
4 \sum_{j =j_{k-1}+1}^{j_{k}}\eps^r_{j} \big|(f^r)'_{j}\big| \big|
(v_h^r)'_{j}\big| \bigg] 
\nonumber\\
\le & \sum_{k \in \Ks_c}\frac{1}{4}\hat{h}_{k}^2 \bigg[
\|{\fb^r}''\|_{\ell^2_{\bar{\eps}^r}(\Ds_k^2)}
\|\vb_h^r\|_{\ell^2_{\bar{\eps}^r}(\Ds_k^2)}
+4\|{\fb^r}'\|_{\ell^2_{\epsb^r(\Ds_k^1)}}
\|(\vb_h^r)'\|_{\ell^2_{\epsb^r (\Ds_k^1)}}\bigg] 
\nonumber\\
\le& \frac{1}{4}\bigg[ \sum_{k \in \Ks_c} \hat{h}_{k}^4
\|{\fb^r}''\|^2_{\ell^2_{\bar{\epsb}^r}(\Ds_k^2)}\bigg]^{\frac{1}{2}}
\|\vb_h^r\|_{\ell^2_{\bar{\epsb}^r}} +\bigg[ \sum_{k \in \Ks_c}\hat{h}_{k}^4
\|{\fb^r}'\|^2_{\ell^2_{\epsb^r (\Ds_k^1)}}\bigg]^{\frac{1}{2}}
\|(\vb_h^r)'\|_{\ell^2_{\epsb^r}}.
\end{align}

For further estimate, we first bound
$\|\vb_h^r\|_{\ell^2_{\bar{\epsb}^r}}$ by
$\|(\vb_h^r)'\|_{\ell^2_{\epsb^r}}$. Since $v_h(x)$ is piecewise linear
with respect to $\Ts^r$ partition, we can apply the trapezoidal
rule to the integration on each element to get
\begin{equation}
\sum_{j=1}^n \bar{\eps}^r_j (v_h^r)_j = \sum_{j=1}^n \frac{1}{2}
(\eps^r_j + \eps^r_{j-1}) (v^r_h)_j = \sum_{j=1}^n \eps^r_j
\frac{1}{2}\big[ (v^r_h)_j + (v^r_h)_{j+1} \big] = \int_0^1 v_h(x) \dx=0.
\end{equation}
The last equality holds by the periodic condition on $v_h$. Thus, we can apply
Lemma \ref{Lemma:DiscPoincareNonUniMesh} in Appendix \ref{AppendixB} and Riez-Thorin Theorem to obtain
\begin{equation}
\|\vb_h^r\|_{\ell^2_{\bar{\eps}^r}} \le \frac{1}{2} n\eps \bigg(\sum_{j=1}^n
\eps^r_j {(v_h^r)'_j}^2\bigg)^{\frac{1}{2}}.
\end{equation}
Since $v_h'(x) = (v_h^r)'_j$ on $(x^r_{j-1}, x^r_j)$, 
\begin{equation}
\sum_{j=1}^n
\eps^r_j {(v_h^r)'_j}^2 = \int_0^1 (v_h')^2 \dx = \| v_h'\|_{L^2[0,1]}.
\end{equation}
By Lemma \ref{lemmaA1} in Appendix \ref{AppendixA}, 
\begin{equation}
\| v_h'\|_{L^2[0,1]} \le \|v'\|_{L^2[0,1]} = \|\vb'\|_{\ell^2_\eps}.
\end{equation}
Combine these results, the estimate stated in the
theorem is easy to establish.
\end{proof}

Having the three lemma and distribute the contribution to each
element, we now give the theorem which essentially gives the
estimate of the residual due to the external force.
\begin{theorem}
\label{Theo:RisidualExternalForce}
For $f,v \in \Us$ and $J_\Us$ defined in \eqref{definitionJv}, we have
\begin{equation}
\label{ResidualExternalForce}
\|\<f,J_\Us \cdot\>_h - \<f,\cdot\>_\eps\|_{\Us^{-1,2}} \le \big\{ \sum_{k \in \Ks_c}
{\eta_k^f}^2 \big\}^{\frac{1}{2}} =: \mathscr{E}_{\rm ext}(f),
\end{equation}
where
\begin{align}
\eta_k^f =& \bigg\{ \frac{1}{128}\big[ \eps^3 (f'_{\ell_{k-1}+1})^2 + \eps^3 (f'_{\ell_{k}+1})^2 \big]^2 
+  \tilde{h}_k^2 \|\fb^r\|^2_{\ell^2_{\bar{\epsb}^r}
  (\Ds^2_k)} \nonumber \\
&+ \frac{1}{64}(n\eps)^4\hat{h}_{k+1}^ 4
\|{\fb^r}''\|^2_{\ell^2_{\bar{\epsb}^r}(\Ds^2_k)} + \hat{h}_{k+1}^4
\|{\fb^r}'\|^2_{\ell^2_{\epsb^r(\Ds^1_k)}} \bigg\}^{\frac{1}{2}},
\end{align}
and $\Ks_c$ is defined in \eqref{Set:K_c}, $\Ds_k^2$ is defined in Lemma \ref{Lemma:ResGraExEn-1},
$\tilde{h}_k$, $\Ds_k^1$ is defined in Lemma
\ref{Lemma:ResGraExEn-2}, and $\hat{h}_{k+1}$
is defined in Lemma \ref{Lemma:ResGraExEn-3}. 
\end{theorem}
\begin{proof}
We can not directly apply the three lemma to estimate the three parts
in \eqref{Eq:ResGraExEnSplit}. The reason is that $J_{\Us} v \neq v_h$, which is the direct interpolation of $v$ according to $\Ts^h$. The way to
circumvent this difficulty is by defining 
\begin{equation}
\label{Eq:Changevtow}
w := v -\sum_{k =
  1}^K\frac{1}{2}(x_{k+1}-x_{k-1})v(x_k) \text{ and
} w_h := I_h w =  J_{\Us} v,
\end{equation}
and noting that 
\begin{equation}
\label{ForceResidualw}
\<f,J_\Us v\>_h - \<f,v\>_\eps = \<f,  w_h\>_h -\<f, w\>_\eps - \<f,
C\>_\eps = \<f,  w_h\>_h -\<f, w\>_\eps,
\end{equation}
and $w'(x) = v'(x) \ \forall x \in \R$. Then by the three lemma, we have
\begin{equation}
\big|\<f,J_\Us v\>_h - \<f,v\>_\eps\big| = \big|\<f,  w_h\>_h -\<f,
w\>_\eps\big| \le \big\{ \sum_{k \in \Ks_c}
{\eta_k^f}^2 \big\}^{\frac{1}{2}} \|w'\|_{L^2[0,1]} = \big\{ \sum_{k \in \Ks_c}
{\eta_k^f}^2 \big\}^{\frac{1}{2}} \|v'\|_{L^2[0,1]},
\end{equation}
which establishes the estimate in the theorem.
\end{proof}

\section{Stability}
\label{Sec:Stability}
Stability of the QC approximation is the second key ingredient for
deriving an a posteriori error bounds. Since we would like to bound
the error of the deformation gradient in $L^2$-norm, we derive the $L^2$
stability estimate in this section. The procedure of deriving the a
posteriori stability condition largely follows that of the a priori
stability condition in \cite{Wang:2010a}.

For an a posteriori error analysis, the natural notion of stability
for energy minimization problem is the coercivity(or, positivity) of
the atomistic Hessian at the projected QC solution $J_{\Us_{\qc}}
y_{\qc}$: 
\begin{equation}
E''_\a(J_{\Us_{\qc}} y_{\qc}) [v,v] \ge c_\a(y_{\qc}) \| v\|_{L^2[0,1]}^2 \quad
\forall v \in \Us,
\label{Stability-1}
\end{equation}
for some constant $c_\a(y_{\qc}) >0$. To avoid notational difficulty, we
vectorize the above inequality and
work on $\Us^\eps$ instead. Let $J_{\Us_{\qc}} \yb^{\qc}$ be the
vectorization of $J_{\Us_{\qc}} y_{\qc}$, then \eqref{Stability-1} is equivalent to 
\begin{equation}
E''_\a(J_{\Us_{\qc}} \yb^{\qc}) [\vb,\vb] \ge c_\a(J_{\Us_{\qc}} \yb^{\qc}) \| \vb \|_{\ell^2_{\epsb}}^2 \quad
\forall \vb \in \Us.
\label{Stability-2}
\end{equation}
In the remainder of this section,
we derive the explicit condition on $\yb^{\qc}$ such that
\eqref{Stability-2} holds.

The Hessian operator of the atomistic model is given by
\begin{equation*}
E''_\a(\yb)[\vb,\vb] = \eps \sum_{\ell=1}^N \phi''(y_\ell')|v_\ell'|^2 + \eps
\sum_{\ell=1}^N \phi''(y_\ell' + y_{\ell+1}') |v_\ell'
+ v'_{\ell+1}|^2 \quad \forall y \in \Ys.
\end{equation*} 
We note that the 'non-local' Hessian terms $|v'_\ell
+ v'_{\ell+1}|^2$ can be rewritten in terms of the 'local' terms
$|v'_\ell|^2$ and $|v'_{\ell+1}|^2$ and a strain-gradient
correction,
\begin{equation*}
|v'_\ell + v'_{\ell+1}|^2 = 2|v'_\ell|^2+2
|v'_{\ell+1}|^2 - \eps^2 |v''_\ell|^2. 
\end{equation*}
Using this formula, we can rewrite the Hessian in the form
\begin{equation*}
E_\a''(\yb)[\vb,\vb] = \eps \sum_{\ell=1}^N A_\ell |v_\ell'|^2 + \eps
\sum_{\ell = 1}^{N} B_\ell |v_\ell''|^2,
\end{equation*}
where
\begin{align}
&A_{\ell}(\yb) = \phi''(y_\ell')+2\phi''(y_{\ell-1}' + y_\ell') +
2\phi''(y_\ell'+y_{\ell+1}')
\label{StabilityConstant-1}
\\
& B_{\ell}(\yb) = -\phi''(y'_{\ell}+y'_{\ell+1}) \nonumber.
\end{align}

Recall our assumption that $\phi$ is convex in $(0, {r_\ast})$ and
concave in $({r_\ast}, +\infty)$. For typical pair interaction
potentials, $y'_{\ell} < {r_\ast}/2$ can only be achieved under
extreme compressive forces. Since, under such extreme conditions a
pair potential may be an inappropriate model to employ anyhow, it is
not too restrictive to assume that the the projected QC solution $J_{\Us_{\qc}}\yb^{\qc}$
satisfies
\begin{displaymath}
(J_{\Us_{\qc}}y^{\qc})'_\ell \ge {r_\ast}/2  \qquad  \forall \ell \in \Z.
\end{displaymath}
As a result of this assumption, and the properties of $\phi$, we have
$-\phi''(y'_{\ell}+y'_{\ell+1}) \ge 0 \ \forall \ell \in \Z$ and thus $B_\ell \ge 0 \  \forall \ell \in \Z$.

As an immediate consequence we obtain the following lemma, which gives
sufficient conditions under which the a posteriori stability of QC approximation can be
guaranteed.
\begin{lemma}
  \label{th:stab_lemma}
  Let $J_{\Us_{\qc}} \yb^{\qc} \in \Ys^\eps$ satisfies $\min_{\ell} (J_{\Us_{\qc}})'_{\ell} \ge {r_\ast}/2$;
  then, 
  \begin{displaymath}
    E''_\a(J_{\Us_{\qc}}\yb^{\qc})[\vb,\vb] \ge A_{\ast}(J_{\Us_{\qc}}\yb^{\qc}) \| \vb'
    \|_{\ell_{\eps}^2}^2 \qquad \forall \vb \in \Us, \quad
    \text{where} \quad
    A_\ast(J_{\Us_{\qc}}\yb^{\qc}) = \min_{\ell = 1, \dots, N} A_\ell(J_{\Us_{\qc}}\yb^{\qc}).
  \end{displaymath}
  The coefficients $A_\ell(J_{\Us_{\qc}}\yb^{\qc})$ are defined in
  \eqref{StabilityConstant-1}.
\end{lemma}

\begin{proof}
If $\min_{\ell} (J_{\Us_{\qc}}y^{\qc})'_{\ell} \ge {r_\ast}/2$, then
\begin{align*}
\qquad  E''_\a(J_{\Us_{\qc}} \yb^{\qc})[\vb,\vb] 
\ge \eps \sum_{\ell=1}^{N} A_\ell(J_{\Us_{\qc}}\yb^{\qc}) |v'_{\ell}|^2
\ge A_{\ast}(J_{\Us_{\qc}}\yb^{\qc}) \eps \sum_{\ell=1}^{N}|v'_{\ell}|^2
= A_{\ast}(J_{\Us_{\qc}}\yb^{\qc}) \| \vb' \|_{\ell^2_\eps}^2.  \qquad \qquad \qquad
\qedhere
\end{align*}
\end{proof}

Before we present the main theorem and its proof in the next section, we
state a useful auxiliary result: a local Lipschitz bound on
$E''_\a$. The proof of this Lipschitz bound is straightforward and is
therefore omitted.
\begin{lemma}
  \label{th:lip_hess}
  Let $\yb, \zb \in \Ys^{\eps}$ such
  that $\min_\ell y_\ell' \geq \mu$ and $\min_\ell z_\ell' \geq \mu$
  for some constant $\mu > 0$, then
 \begin{displaymath}
    \big| \{\Es_\a''(\yb) - \Es_\a''(\zb)\}[\vb, \wb] \big|
    \leq C_{\rm Lip} \| \yb' - \zb' \|_{\ell_\eps^\infty} \| \vb'\|_{\ell^2_\eps}
    \|\wb'\|_{\ell^2_\eps}
    \qquad \forall \vb, \wb \in \Us,
  \end{displaymath}
  where $C_{\rm Lip} = M_3([\mu,+\infty)) + {8 M_3([2\mu,
  +\infty))}$ and $M_i(S) = \max_{\xi \in \Ss} |\phi^i(\Ss)|$.
\end{lemma}

\section{A Posteriori Error Estimates}
\label{Sec:A_Posteriori_Error}
\subsection{The a posterior error estimates for the deformation gradient}
The error we estimate is $e := y_\a - J_{\Us_{\qc}}y_{\qc}$
in the $\Us^{1,2}$-norm, as for $y_1, y_2 \in \Ys$, $y_1 - y_2 \in \Us$. To avoid
technicalities associated with the nonlinearity of our models, we make
an a priori assumption: we assume the existence of the atomistic and QC
solutions and make a mild requirement on their smoothness and
closeness (cf. \eqref{eq:w1inf_apriori_assmpt}).

\begin{theorem}
  Let $y_{\qc}$ be a solution of the QC problem \eqref{LocalMinimumQC} whose gradients are such
  that $\min_\ell \big(J_{\Us_{\qc}} y^{\qc} \big)'_\ell \geq r_* / 2 \ \forall
  \ell \in \Z$ and $A_*(J_{\Us_{\qc}}\yb^{\qc}) >
  0$, where $A_*$ is defined in the statement of Lemma
  \ref{th:stab_lemma}. Suppose, further, that $y^{\a}$ is a solution
  of the atomistic model \eqref{Eq:LocalMinimaA} such that, for some $\tau
  > 0$,
  \begin{equation}
    \label{eq:w1inf_apriori_assmpt}
    \| (y^{\a} - J_{\Us_{\qc}} y_{\qc})' \|_{L^\infty [0,1]} =  \| (\yb^{\a} - J_{\Us_{\qc}} \yb^{\qc})' \|_{\ell^\infty_\eps}\leq \tau.
  \end{equation}
Then, if $\tau$ is sufficiently small, we have the error estimate
  \begin{equation}
    \label{eq:main_errest}
    \| y^\a - J_{\Us_{\qc}} y_{\qc} \|_{L^2[0,1]} = \| (\yb^{\a} - J_{\Us_{\qc}} \yb^{\qc})'
    \|_{\ell^2_\eps} =\leq {\smfrac{2}{A_*(\yb^{\a})}} \big(
    \mathscr{E}_{\rm store}(y_{\qc})+\mathscr{E}_{\rm ext}(f) \big),
  \end{equation}
where the functional of the residual of the stored energy $\mathscr{E}_{\rm store}(\cdot)$ is defined in
\eqref{ResidualStoredEnergy} and the functional of the approximation error for the external forces
$\mathscr{E}_{\rm ext}(\cdot)$ is defined in \eqref{ResidualExternalForce}.
\end{theorem}

\begin{proof}
From the mean value theorem we deduce that there exists $\thetab \in
\conv\{ \yb^{\a}, J_{\Us_{\qc}} \yb^{\qc}\}$ such that
\begin{align*}
E''_\a(\thetab)[\eb,\eb] = & E'_\a(\yb^{\a})[\vb] - E'_{\qc}(J_{\Us_{\qc}} \yb^{\qc})[ J_{\Us} \vb]
\nonumber\\
                                 = & \big( \Es'_\a(\yb^{\a})[\vb] -
                                 \Es'_\a(J_{\Us_{\qc}}\yb^{\qc})[J_{\Us} \vb] \big) \nonumber\\
                                   -& \big( \< \fb, \vb\>_\eps -
                                   \<\fb, J_{\Us}\vb\>_h \big). \nonumber\\
\end{align*}
The first group was analyzed in section \ref{SectionResidualStoredEnergy} Theorem \ref{Theo:ResidualStoredEnergy} and the second
group was analyzed in section \ref{SectionResidualExternalForce} Theorem \ref{Theo:RisidualExternalForce}. Inserting these
estimates we arrive at
\begin{equation}
\label{ErrorEstimatShort}
E''_\a(\thetab)[\eb,\eb] \le \big(\mathscr{E}_{\rm store}(y^{\qc}) +
\mathscr{E}_{\rm ext}(f) \big) \| \eb'\|_{\ell^2_\eps}.
\end{equation}

It remains to prove a lower bound on $\Es''_\a (\thetab)
[\eb,\eb]$. From our assumption that $\min (J_{\Us_{\qc}} \yb^{\qc})'_\ell
\ge r_* / 2$, and from \eqref{eq:w1inf_apriori_assmpt} it follows that
\begin{displaymath}
\min_{\ell} \theta'_\ell \ge r_* / 2-\tau.
\end{displaymath}
Assuming that $\tau$ is sufficiently small, e.g., $\tau \le \tau_1:=
\frac{1}{4} \min_\ell (J_{\Us_{\qc}} \yb^{\qc})'_\ell$, we can apply Lemma
\ref{th:lip_hess} to deduce that
\begin{align}
\label{StabilityFinalShort}
\Es''_\a(\thetab)[\eb,\eb] \ge & \Es''_\a(J_{\Us_{\qc}} \yb^{\qc})[\eb,\eb] -
C_{\rm Lip} \| (\thetab - \yb^{\a}) \|_{\ell^\infty_\eps}
\|\eb'\|^2_{\ell^2_\eps} \nonumber \\
\ge & \Es''_\a(J_{\Us_{\qc}} \yb^{\qc})[\eb,\eb] -
C_{\rm Lip} \tau \|\eb'\|^2_{\ell^2_\eps},
\end{align}
where $C_{\rm Lip}$ may depend on $\tau_1$.

We can now apply our stability analysis in Section
\ref{Sec:Stability}. Since $(J_{\Us_{\qc}} \yb^{\qc})_\ell' \geq r_*/2$ for all
  $\ell$, Lemma \ref{th:stab_lemma}  implies that
  \begin{displaymath}
    \Es''_\a(J_{\Us_{\qc}} \yb^{\qc})[\eb,\eb]  \geq A_*(J_{\Us_{\qc}} \yb^{\qc})
    \| \eb'\|_{\ell^2_\eps}^2,
  \end{displaymath}
  which, combined with \eqref{ErrorEstimatShort} and \eqref{StabilityFinalShort}
  , yields
  \begin{displaymath}
    \big(A_*(J_{\Us_{\qc}} \yb^{\qc}) - C_{\rm Lip} \tau \big) \| \eb' \|_{\ell^2_\eps}^2 \leq \Es''_\a(\theta)[ \eb,  \eb]
    \leq \big(\mathscr{E}_{\rm store}(y_{\qc}) +
\mathscr{E}_{\rm ext}(f) \big) \| \eb'\|_{\ell^2_\eps}.
  \end{displaymath}
  Dividing through by $\| \eb' \|_{\ell^2_\eps}$, and assuming
  that $\tau \leq \min(\tau_1, \tau_2)$ where $\tau_2 = A_*(J_{\Us_{\qc}}\yb^{\qc})
  / (2C_{\rm Lip})$, we deduce that
  \begin{displaymath}
    \smfrac{A_*(J_{\Us_{\qc}}\yb^{\qc})}{2} \| \eb'
    \|_{\ell^2_\eps} \leq  \big(\mathscr{E}_{\rm store}(y_{\qc}) +
\mathscr{E}_{\rm ext}(f) \big),
  \end{displaymath}
  which concludes the proof of the a posteriori error estimate for the
  deformation gradient.
\end{proof}

\subsection{The a posterior error estimate for the energy}
Besides the deformation gradient, the energy of the system is another
quantity of interest. In this section, we derive an a posteriori
error estimator for the energy difference between the atomistic model
and the QC approximation, namely,
\begin{equation}
E_\a (y^{\a}) - E_{\qc} (y_{\qc}) .
\label{EnergyDifference}
\end{equation}

To analyze this difference, we decompose \eqref{EnergyDifference} as
\begin{equation}
\big|E_\a(y^{\a}) - E_{\qc} (y_{\qc})\big| = \big|E_\a(y^{\a}) - E_\a(J_{\Us_{\qc}} y_{\qc})\big|
+ \big| E_\a(J_{\Us_{\qc}} y_{\qc}) - E_{\qc} (y_{\qc}) \big|.
\end{equation}
We then analyze the two groups separately.

To analyze the first group, we have the following Lemma.

\begin{lemma}
  Let $y, z \in \Ys$ and $\yb, \zb \in \Ys^\eps$ be their
  vectorizations, such
  that $\min_\ell y_\ell' \geq \mu$ and $\min_\ell z_\ell' \geq \mu$
  for some constant $\mu > 0$, and $y \in \argmin E_\a(\Ys)$. Let $\eb = \yb - \zb$, then
\begin{equation}
\big| E_\a(\yb) - E_\a(\zb) \big| \le C^E_{\rm Lip} \| \eb'\|^2_{\ell^2_\eps},
\end{equation}
where $C^E_{\rm Lip} = \frac{1}{2}M_2([\mu,+\infty)) + {2 M_3([2\mu,
  +\infty))}$, where $M_i(S) = \max_{\xi \in \Ss} |\phi^i(\Ss)|$.
\end{lemma}

\begin{proof}
We first rewrite the difference of the total energy as the summation
of the differences of the stored energy and that of the external energy:
\begin{displaymath}
E_\a(\yb) - E_\a(\zb) = \big(\Es_\a(\yb) - \Es_\a(\zb)\big) - \big(\<\fb,\zb\>_\eps -
\<\fb,\yb\>_\eps \big)
\end{displaymath}
For the difference of the stored energy, we have
\begin{align*}
\Es_\a(\yb) - \Es_\a(\zb) = & \eps\sum_{\ell=1}^N\big[ \phi(y'_\ell) -
\phi(z'_\ell) \big] + \eps \sum_{\ell=1}^N \big[
\phi(y'_\ell+y'_{\ell+1}) - \phi(z'_\ell+z'_{\ell+1}) \big]
\nonumber\\
                             = &\eps \sum_{\ell=1}^N \big[
                             \phi'(y'_\ell) e'_\ell -
                             \frac{1}{2}\phi''(\xi^1_\ell){e'_\ell}^2\big]
                             \nonumber\\
                             +&\eps \sum_{\ell=1}^N \big[
                             \phi'(y'_\ell+y'_{\ell+1})
                             (e'_\ell+e'_{\ell+1}) -
                             \frac{1}{2}\phi''(\xi^2_\ell)(e'_\ell+e'_{\ell+1})^2
                             \big]\nonumber \\
                             =& \Es_\a'(\yb)[\eb] - \frac{1}{2} \eps \sum_{\ell=1}^N
                             \phi''(\xi^1_\ell) {e'_\ell}^2 - \frac{1}{2}\eps \sum_{\ell=1}^N \phi''(\xi^2_\ell) (e'_\ell+e'_{\ell+1})^2,
\end{align*}
where $\xi^1_\ell \in \conv\{y'_\ell, z'_\ell\}$ and $\xi^2_\ell \in \conv\{y'_\ell+y'_{\ell+1}, z'_\ell+z'_{\ell+1}\}$.

For the difference of the energy caused by the external forces, we have
\begin{equation*}
\<\fb,\zb\>_\eps - \<\fb,\yb\>_\eps = -\<\fb,\eb\>_\eps = -\Es_\a'(\yb)[\eb],
\end{equation*}
by the first optimality condition of $y \in \argmin E_\a(\Ys)$. It is
then easy to obtain the estimate stated in the Lemma by using
Cauchy-Schwaz inequality to the non-local term. 
\end{proof}

\begin{lemma}
For $y_h \in \Ys_{\qc}$ and $y'_h(x) > 0$, we have
\begin{equation}
\big| E_\a(J_{\Us_{\qc}} y_h) - E_{\qc}(y_h) \big| \le \sum_{k \in \Ks_c}
{\eta_E^e}_k + {\eta_E^f}_k,
\end{equation}
where 
\begin{equation}
\eta_k^e = \frac{1}{2} \sum_{j = -1}^{1} [[\phi]]_{\ell_{k-1}+j} + \frac{1}{2}\sum_{j = -1}^{1} [[\phi
]]_{\ell_{k}+j},
\end{equation}
if $k \in \Ks'_c$,
\begin{equation}
\eta_k^e =  \sum_{j = -1}^1  [[ \phi
]]_{\ell_{La_i}+j} + \frac{1}{2} \sum_{j=-1}^1[[ \phi
]]_{\ell_{k-1}+j},
\end{equation}
if $k = L_{a_i}$ for some $i \in \{1,2,\ldots,M\}$, and
\begin{equation}
\eta_k^e =  \sum_{j = -1}^1  [[ \phi
]]_{\ell_{Ra_i}+j} + \frac{1}{2} \sum_{j=-1}^1[[ \phi
]]_{\ell_{k}+j},
\end{equation}
if $k = R_{a_i}$ for some $i \in \{1,2,\ldots,M\}$. $[[\phi]]_{\ell_k}$'s and ${\eta_E^f}_k$'s will be defined in the proof.
\end{lemma}

\begin{proof}
We first decompose the energy difference to two parts:
\begin{displaymath}
E_\a(J_{\Us_{\qc}} y_h) - E_{\qc}(y_h) = \big(\Es_\a(J_{\Us_{\qc}} y_h) -
\Es_{\qc}(y_h)\big) - \big(\<f, J_{\Us_{\qc}}y_h\>_\eps -
\<f_h,y_h\>_h \big).
\end{displaymath}

We first analyze the energy difference of the stored energy. Since $r_b D_bJ_{\Us_{\qc}} y_h = r_b D_b y_h$, we have
\begin{equation}
\Es_\a(J_{\Us_{\qc}} y_h) =  \sum_{b \in \Bs} a_b(J_{\Us_{\qc}} y_h)  
= \sum_{b \in \Bs} \eps \phi(r_b D_bJ_{\Us_{\qc}} y_h) = \sum_{b \in \Bs} \eps \phi(r_b D_b y_h),
\end{equation}
and
\begin{equation}
\Es_{\qc}(y_h) = \sum_{b \in \Bs} \big[ a_b(y_h) + c_b(y_h) \big]
= \sum_{b \in \Bs}\bigg[ \frac{|b \cap \Omega_\a|}{r_b} \phi\big(r_b D_{b \cap
  \Omega_\a} y \big) + \frac{1}{r_b} \int_{b\cap \Omega_c}
\phi(\nabla_{r_b}y(x))\dx \bigg].
\end{equation}

We analyze the energy difference bond by bond,

If $b \subset \Omega_\a $, then 
\begin{displaymath}
a_b(y_h) + c_b(y_h) = a_b(y_h) = \eps \phi(r_b D_b y_h) =  a_b(J_{\Us_{\qc}} y_h),
\end{displaymath}
and the energy difference in this bond is thus $0$.

If $b \subset \Omega_c \cap T_k$ for some $k \in \Ks_c$, then $|b\cap\Omega_\a| = 0$ and
\begin{align}
\label{Eq:EnergyErrorBondinC}
a_b(J_{\Us_{\qc}} y_h) - \big[ a_b(y_h) + c_b(y_h) \big] =& a_b(J_{\Us_{\qc}} y_h) - c_b(y_h)\nonumber\\
= & \eps
\phi(r_b D_b y_h) - \frac{1}{r_b} \int_{b\cap \Omega_c}
\phi(\nabla_{r_b}y_h(x))\dx \nonumber \\
=& \frac{1}{r_b} \int_b \big[ \phi(r_b D_b y_h) -
\phi(\nabla_{r_b}y_h(x)) \big] \dx.
\end{align}
Since $y_h$ is affine on $T_k$, $\nabla_{r_b}y(x))=r_b D_b y_h$ and subsequently, \eqref{Eq:EnergyErrorBondinC} is $0$.

We are left with the interaction bonds crossing the atomistic-continuum interface and the boundaries of the elements in the continuum region. 
Again because of its tediousness, we leave the detail of this analysis to the Appendix and only give the results here:
\begin{align}
\Es_\a(J_{\Us_{\qc}} y_h) - \Es_{\qc}(y_h) = \sum_{i = 1}^M \sum_{j = -1}^1  [[ \phi
]]_{\ell_{La_i}+j} + \sum_{i = 1}^M \sum_{j = -1}^1  [[ \phi
]]_{\ell_{Ra_i}+j} + \sum_{k \in \Ks'_c} \sum_{j=-1}^1[[ \phi
]]_{\ell_{k}+j}.
\end{align}
For $k = L_{a_i}$ where $i \in \{1,2,\ldots,M\}$, we have
\begin{align}
[[\phi]]_{\ell_{k}} =&\eps \big\{ \phi( (1-\theta_k) y'_h|_{T_{k+1}} + \theta_k
y'_h|_{T_k}) - (1-\theta_k) \phi( y_h'|_{T_{k+1}}) -\theta_k\phi( y_h'|_{T_{k}})\big\},
\end{align}
\begin{align}
[[\phi]]_{\ell_{k-1}} =2\eps \big\{ \phi( (1-\theta_k) y_h'|_{T_{k+1}} + (1+\theta_k)\theta_k y_h'|_{T_k}) - (1-\theta_k) \phi(2y_h'|_{T_{k+1}}) - (1+\theta_k)\phi(2y_h'|_{T_{k}}) \big\}.
\end{align}
and
\begin{align}
[[\phi]]_{\ell_{k+1}} =&2\eps \big\{ \phi( y_h'|_{T_{k+2}}+\theta_k y_h'|_{T_k} +(1-\theta_k) y_h'|_{T_{k+1}} ) \nonumber \\
&-(2-\theta_k) \phi\big( \frac{2}{2-\theta_k}y_h'|_{T_{k+2}} + \frac{2(1-\theta_k)}{2-\theta_k} y_h'|_{T_{k+1}}\big) - \theta_k\phi(2y_h'|_{T_{k}}) \big\}.
\end{align}

For $k = R_{a_i}$ where $i \in \{1,2,\ldots,M\}$, we have
\begin{align}
[[\phi]]_{\ell_{k}} = \eps \big\{ \phi( (1-\theta_k) y'_h|_{T_{k+1}} + \theta_k
y'_h|_{T_k}) - \theta_k\phi( y_h'|_{T_{k}}) -(1-\theta_k) \phi( y_h'|_{T_{k+1}}) \big\},
\end{align}
\begin{align}
[[\phi]]_{\ell_{k-1}} = &2\eps \big\{ \phi( (1-\theta_k) y_h'|_{T_{k+1}} + \theta_k y_h'|_{T_k} +
y_h'|_{T_{k-1}}) \nonumber\\
&-(1-\theta_k)\phi(2y_h'|_{T_{k+1}}) -(1+\theta_k) \phi\big( \frac{2\theta_k}{1+\theta_k} y_h'|_{T_k} +
\frac{2}{1+\theta_k}y_h'|_{T_{k-1}}\big) \big\},
\end{align}
and 
\begin{align}
[[\phi]]_{\ell_{k+1}} = 2\eps \big\{ \phi( (2-\theta_k)y_h'|_{T_{k+1}} +
\theta_ky_h'|_{T_k}) - (2-\theta_k)\phi(2y_h'|_{T_{k+1}})- \theta_k \phi\big( 2y_h'|_{T_k} \big) \big\}.
\end{align}
For $k \in \Ks'_c$, we have
\begin{align}
[[\phi]]_{\ell_k} = \eps \big\{ \phi((1-\theta_k) y_h'|_{T_{k+1}} + \theta_k
y_h'|_{T_k}) - (1-\theta_k)\phi(y_h'|_{T_{k+1}}) -\theta_k \phi(y_h'|_{T_k}) \big\},
\end{align}
\begin{align}
[[\phi]]_{\ell_k-1} = \frac{1}{2}\eps \big\{ 2\phi((1-\theta_k) y_h'|_{T_{k+1}} +
(1+\theta_k) y_h'|_{T_{k}}) - (1-\theta_k)
\phi(2y_h'|_{T_{k+1}})- (1+\theta_k) \phi(2y_h'|_{T_k})  \big\},
\end{align}
and
\begin{align}
[[\phi]]_{\ell_k+1} = \frac{1}{2}\eps \big\{ 2\phi((2-\theta_k) y_h'|_{T_{k+1}} +
\theta_k y_h'|_{T_{k}}) - \theta_k
\phi(2y_h'|_{T_{k+1}}) - (2-\theta_k) \phi(2y_h'|_{T_k})  \big\}.
\end{align}

We then analyze the energy difference caused by the external
forces. The energy difference is given by
\begin{displaymath}
\<f, J_{\Us_{\qc}}u_h\>_\eps - \<f,u_h\>_h  = \<f, u_h\>_\eps - \<f,u_h\>_h,
\end{displaymath}
since $J_{\Us_{\qc}} u_h = u_h+C$ for some contant $C$ and
$\<f,C\>_\eps = 0 \ \forall C$. We decompose this energy difference to
each element and write it as
\begin{equation}
\big|\<f, u_h\>_\eps - \<f,u_h\>_h \big|\le \sum_{k=1}^{K} {\eta_E^f}_k, 
\end{equation}
where
\begin{align}
{\eta_E^f}_k = &\bigg|(1-\theta_k)\frac{1}{2}\eps(f_{\ell_{k-1}}u_{\ell_{k-1}} +
  f_{\ell_{k-1}+1}u_{\ell_{k-1}+1}) 
  + \frac{1}{2} \sum_{\ell = \ell_{k-1}+2}^{\ell_{k}-1} \eps (f_\ell u_\ell + f_{\ell+1} u_{\ell+1})
  \nonumber \\ 
&+ \theta
  \frac{1}{2} \eps (f_{\ell_{k}} u_{\ell_{k}} + f_{\ell_{k}+1}
  u_{\ell_{k}+1}) \big\} -\sum_{k=1}^K \frac{1}{2}(x_{k}-x_{k-1}) \big[ f(x_{k-1})y(x_{k-1}) +
f(x_{k})y(x_{k}) \big] \bigg|,
\end{align}
where $f_\ell = f(\ell\eps)$ and $u_\ell = u(\ell\eps)$.
\end{proof}

\section{Numerical Experiments}
\label{Numerics}
In this section, we present numerical experiments to illustrate our
analysis. Throughout this section we fix $F = 1$,
$N = 8193$, and let $\phi$ be the Morse potential
\begin{equation*}
  \phi(r) = \exp(-2\alpha(r-1))-2\exp(-\alpha(r-1)),
\end{equation*}
with the parameter $\alpha = 5$. 

For our benchmark problem, we defined the external force $\fb$ to be
\begin{align*}
f_\ell = 
\left\{
\begin{array}{rl}
-0.1\big\{1-\frac{|\ell-\frac{N-1}{2}|}{\frac{N-1}{2}}\big\}\frac{N}{|\ell-\frac{N-1}{2}-0.5|},
&\text{for $\ell \le \frac{N-1}{2}$},\\
 0.1\big\{1-\frac{\ell-\frac{N-1}{2}-1}{\frac{N-1}{2}}\big\}\frac{N}{|\ell-\frac{N-1}{2}-0.5|},
 &\text{for $\ell \ge \frac{N-1}{2}+1$.}
\end{array} \right.
\end{align*}

We briefly explain the meaning of the external force. On each atom, the external
force is a product of three components. The third component, namely
$\frac{N}{|\ell-\frac{N-1}{2}-0.5|}$ is essentially
$\frac{1}{r_\ell}$ where $r_\ell$ is the distance between an atom and
the center of this atomistic chain located at
$\frac{N-1}{2}+0.5$. This non-linear force will create a defect
in the middle of the chain but affect little in the far field. The
second component, namely
$1-\frac{|\ell-\frac{N-1}{2}|}{\frac{N-1}{2}}$, adds a decay
of the first component and in particular, it is $0$ when $\ell =
N$, which prevents the 'kink' of the force on the boundary due to a
rapid change of the sign of the force that will leads to non-smooth
deformation gradient that should be contained in the atomistic region. The first component, which is the constant
$0.1$, is to rescale the force so that the solution of this problem is stable.

We solve for the atomistic problem and consider the solution to be the accurate solution. We then solve for the QC
problem on different meshes generated by the mesh refinement schemes. 

We show two relative errors against the number of degrees of
freedom. The first one is the error of the deformation
gradient in $L_2$-norm over the $L_2$-norm of the difference between the deformation gradient
of the atomistic solution and the homogeneous state, which is defined by 
\begin{equation}
\label{RelativeErrorDeformation}
e_{deformation} := \frac{\|
  y_{\qc}'-y_\a'\|_{L^2[0,1]}}{\|y_\a'- Fx \|_{L^2[0,1]}}.
\end{equation}
 The second relative error is the absolute value of the energy
 difference of the atomistic solution and the QC solution over the
 absolute value of the energy change of the atomistic solution from the homogeneous state, which is defined by
\begin{equation}
\label{RelativeErrorEnergy}
e_{energy} := \frac{|E_\a(y_\a) - E_{\qc}(y_{\qc})|}{|E_\a(y_\a) - E_\a(Fx)|}.
\end{equation}

Before we present the plots of the errors, we first introduce the mesh
generating schemes.

\subsection{Mesh Construction}
To avoid unnecessary technical difficulty in the mesh refinement
algorithm, we assume that the defect core is already captured in the middle of the chain. There are three mesh generating schemes we use. 

The first mesh generating scheme is derived in
Section 7.1 of \cite{OrtShap2010b} using calculus of
variations. From this analysis, we get that the (quasi-)optimal mesh size in
the continuum region, with the restriction that the atomistic region is symmetric and has
$K$ atoms on each side, is given by 
\begin{equation}
\label{OptimalMesh}
h(r) = \big(\frac{f(K\eps)}{f(r)}\frac{r}{K\eps}\big)^{\frac{2}{3}}.
\end{equation}
Since the mesh size can not change continuously and we restrict the
smallest mesh size in the continuum region to be $2\eps$, we use the
following algorithm to generate this mesh (we only list the case on the right hand side of the atomistic region):

\begin{algorithm}
\label{Algo:CHMesh}
\begin{enumerate}
\item Set atom $\frac{N-1}{2} + 1$ to be the middle of the atomistic region.
\item Choose $K$ so that there are $K$ atoms on each side of the
  atomistic region.
\item Choose $h$ to be $2\eps$ for every element on the right hand side
  of the atomistic region until $h(r) > 2\eps$, where $r$ is the
  distance between the right boundary of the previous element and the middle of the atomistic region.
\item Choose $h$ according to \eqref{OptimalMesh} until the right
  boundary of the newly created element is out of the right limit of
  the chain.
\end{enumerate}
\end{algorithm}

The second mesh generating scheme is essentially a mesh refinement
process according to the error estimator with respect to the
deformation gradient according to Lemma \ref{Lemma:ResGraExEn-1},
Lemma \ref{Lemma:ResGraExEn-2} and Lemma \ref{Lemma:ResGraExEn-3}. The mesh
refinement algorithm is stated as follows:

\begin{algorithm}
\label{Algo:EDMesh}
\begin{enumerate}
\item Set atom $\frac{N-1}{2} + 1$ to be the middle of the atomistic region.
\item Choose $5$ atoms on each side of the atomistic region.
\item Divide the left and the right part of the continuum region into
  two equally large element
\item Compute the QC solution on this mesh and then compute the squared error indicator of each element $\eta_i$ and
  sort these indicators according to its value.
\item Bisect the first $M$ sorted elements such that
\begin{equation}
\sum_{i = 1}^{M-1} \eta^2_{i} \le 0.5\eta^2  \text{ and }  \sum_{i = 1}^{M} \eta^2_{i} \ge 0.5\eta^2,
\end{equation}
where 
$\eta_i$ is the error estimator of each element defined by
\begin{equation}
\eta^{deformation}_i = \big[(\eta^e_i)^2 + (\eta^f_i)^2\big])^{\frac{1}{2}} \big/ \smfrac{A_*(J_{\Us_{\qc}}\yb^{\qc})}{2}.
\end{equation}
If the element is
near the atomistic region, merge the element into the atomistic region.
\item If the resulting mesh reaches the maximal number of degrees of
  freedom, stop the process, else, go to Step 4.
\end{enumerate}
\end{algorithm}

The third mesh generating scheme is the mesh refinement
process according to the error estimator with respect to the
energy which is defined by
\begin{equation}
\eta^{energy}_i = C^E_{Lip}\big(\eta^{deformation}_i\big)^2 +
{\eta^e_E}_k + {\eta^f_E}_k ,
\end{equation}
for each element and the refinement algorithm is exactly the same.

In short, the first and second mesh generating schemes tend to
minimize the error in the deformation gradient and the third one tends
to minimize the error in the total energy.

\subsection{Numerical Results}
We compare the relative errors defined in
\ref{RelativeErrorDeformation} and \ref{RelativeErrorEnergy}. We plot
the relative errors against the number of degrees of freedom with
respect to the meshes generated.

\begin{figure}[h]

  \begin{minipage}{10.5cm}
    \begin{center}
      \includegraphics[width=10.5cm, bb = 80 220 600 600]{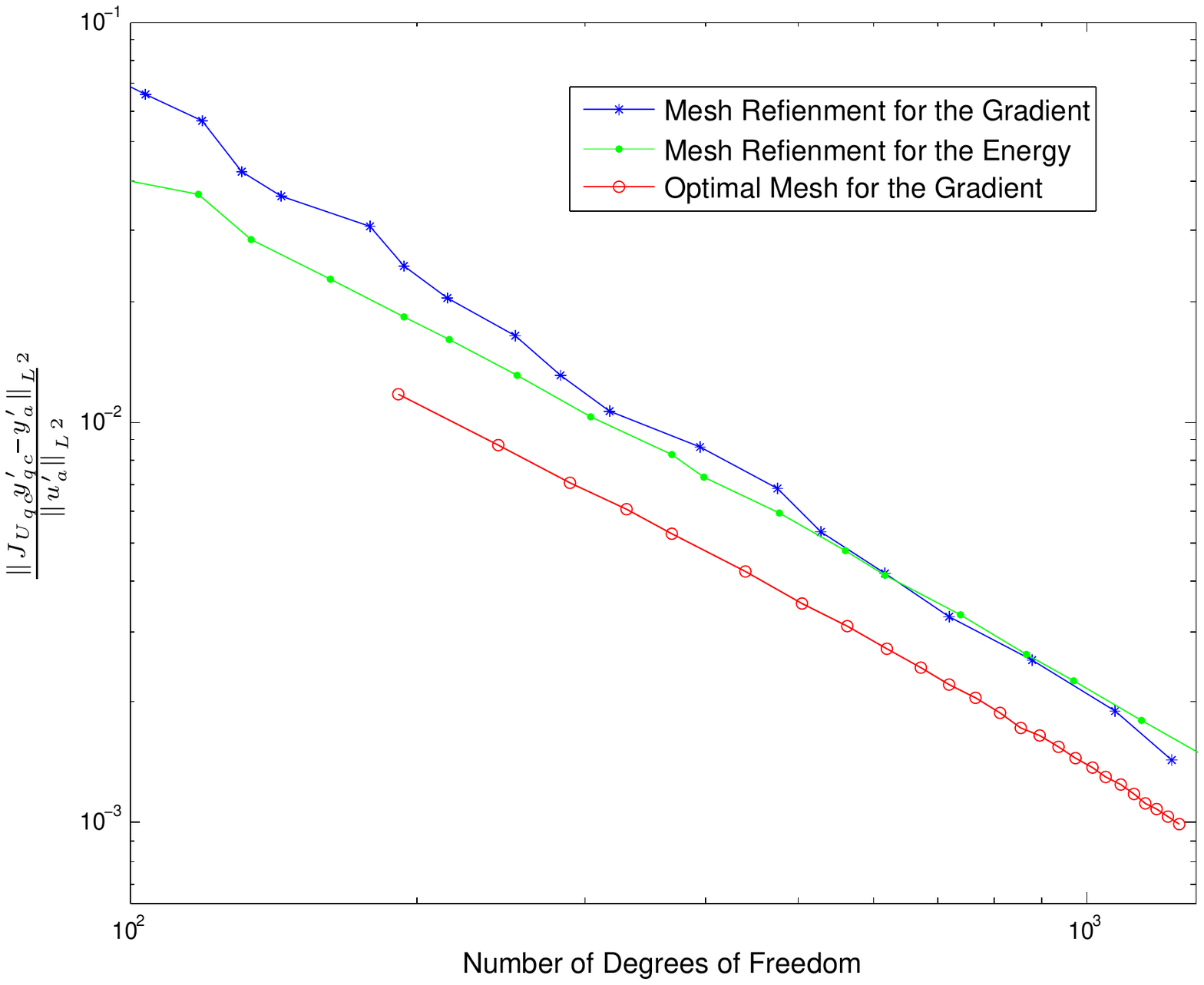}  
      \caption{Relative Error of the Gradient} 
      \label{Fig:Relative_Error_in_Gradient}
    \end{center}
  \end{minipage}

\end{figure}

\begin{figure}[h]
  \begin{minipage}{10.5cm}
    \begin{center}
      \includegraphics[width=10.5cm, bb = 80 220 600 600]{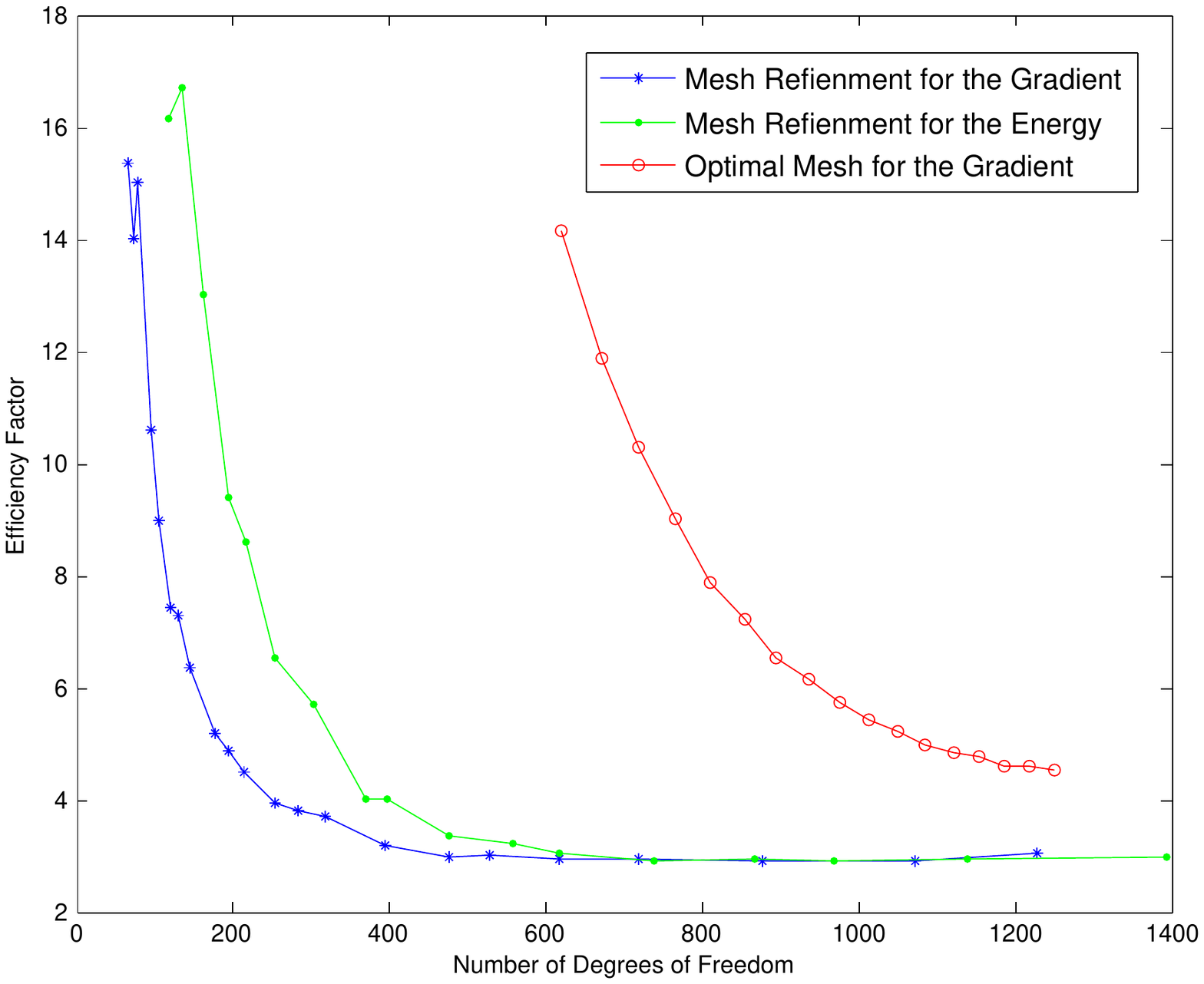}  
      \caption{Efficiency Factor of the Gradient} 
      \label{Fig:Efficiency_Factor_Gradient}
    \end{center}
  \end{minipage}
\end{figure}

Figure \ref{Fig:Relative_Error_in_Gradient} shows that the
pre-defined optimal mesh performs better than the two
mesh refinement strategies for a fix number of degrees of freedom. The
possible reason for this is that, due to some technical difficulty in
coding, both of the mesh refinement algorithms tend to produce larger atomistic
region by merging the elements in the continuum region to the atomistic region and create some unnecessary degrees of freedom. For the two mesh
refinement strategies, the one according to the gradient error indicator perform better asymptotically. 

Figure \ref{Fig:Efficiency_Factor_Gradient} shows the efficiency
factor of the error estimator of the deformation gradient. It shows
that the efficiency factor is comparatively large but decreases as the
number of degrees of freedom increases and
finally become stable. The reason for this phenomenon lies in the form
of the external force. One can show that if the
external force takes the form of $f(r) = \frac{1}{r}$, where $r$ is
the distance to the centre of the defect, then the residual due to the
external force is of order $h^2$ as opposed to order $h$ in general
which is achieved by our analysis. As a result, our estimate
exaggerate the real error by $\frac{1}{h}$ for this particular
external force. This phenomenon gradually disappear as the continuum region moves apart from the centre of the defect since 
the influence of this exaggeration is eliminated as the external force tends to $0$ when it is away from the centre of the defect, which 
makes the residual of the sotred energy become the leading error term. It can also well explain the fact that the efficiency of
the estimate is better for the mesh refinement strategies than the
pre-defined mesh for a certain number of degrees of freedom, as the
two mesh refinement algorithms tend to put more atoms in the atomistic
region, i.e., the continuum region is further away from the centre of
defect than that of the pre-defined mesh.

\begin{figure}[h]

  \begin{minipage}{10.5cm}
  \begin{center}
      \includegraphics[width=10.5cm, bb = 80 220 600 600]{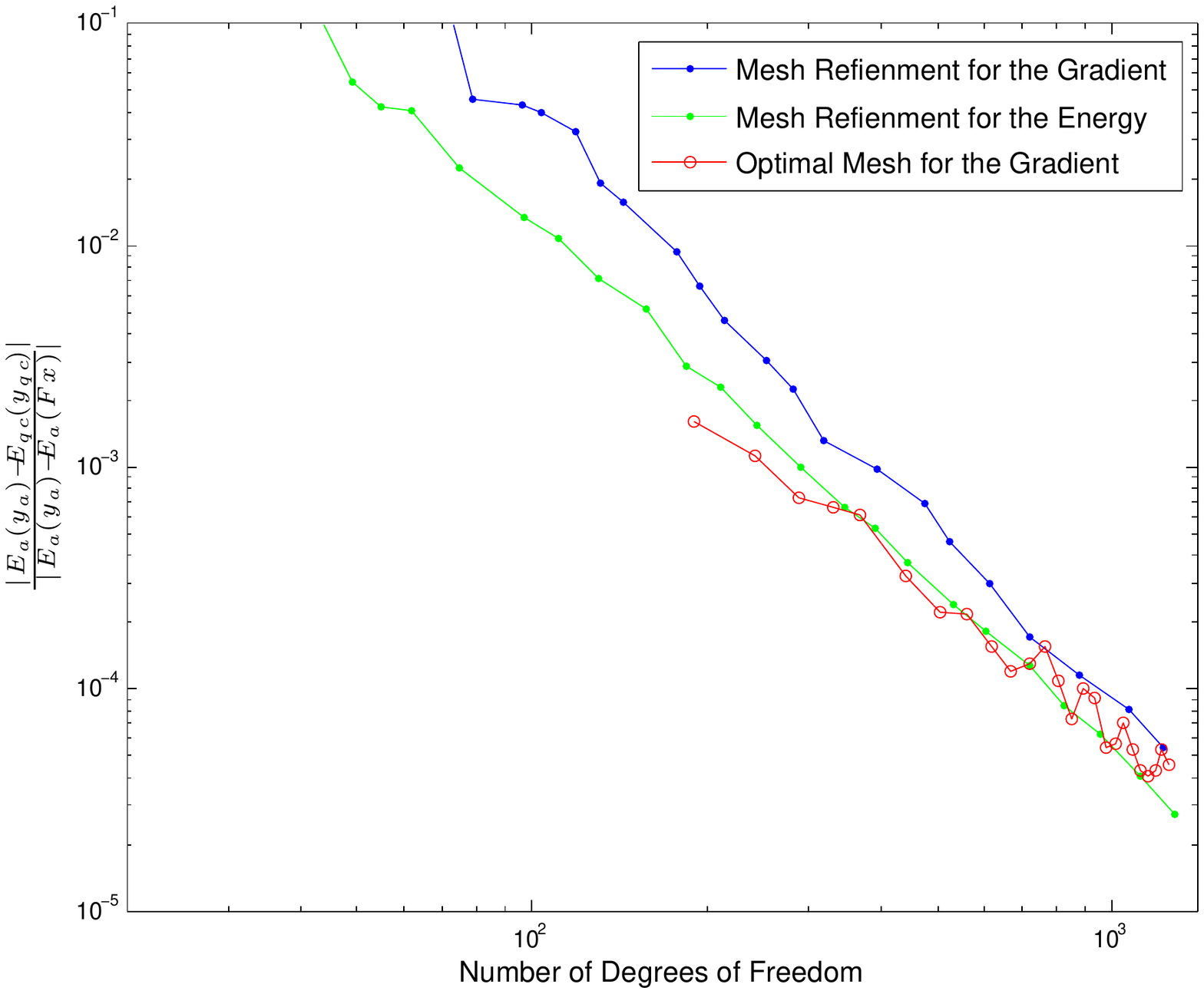}  
      \caption{Relative Error of the Total Energy} 
      \label{Fig:Relative_Error_in_Energy}
 \end{center}
  \end{minipage}
 
\end{figure}

\begin{figure}
  \begin{minipage}{10.5cm}
    \begin{center}
      \includegraphics[width=10.5cm, bb = 80 220 600 600]{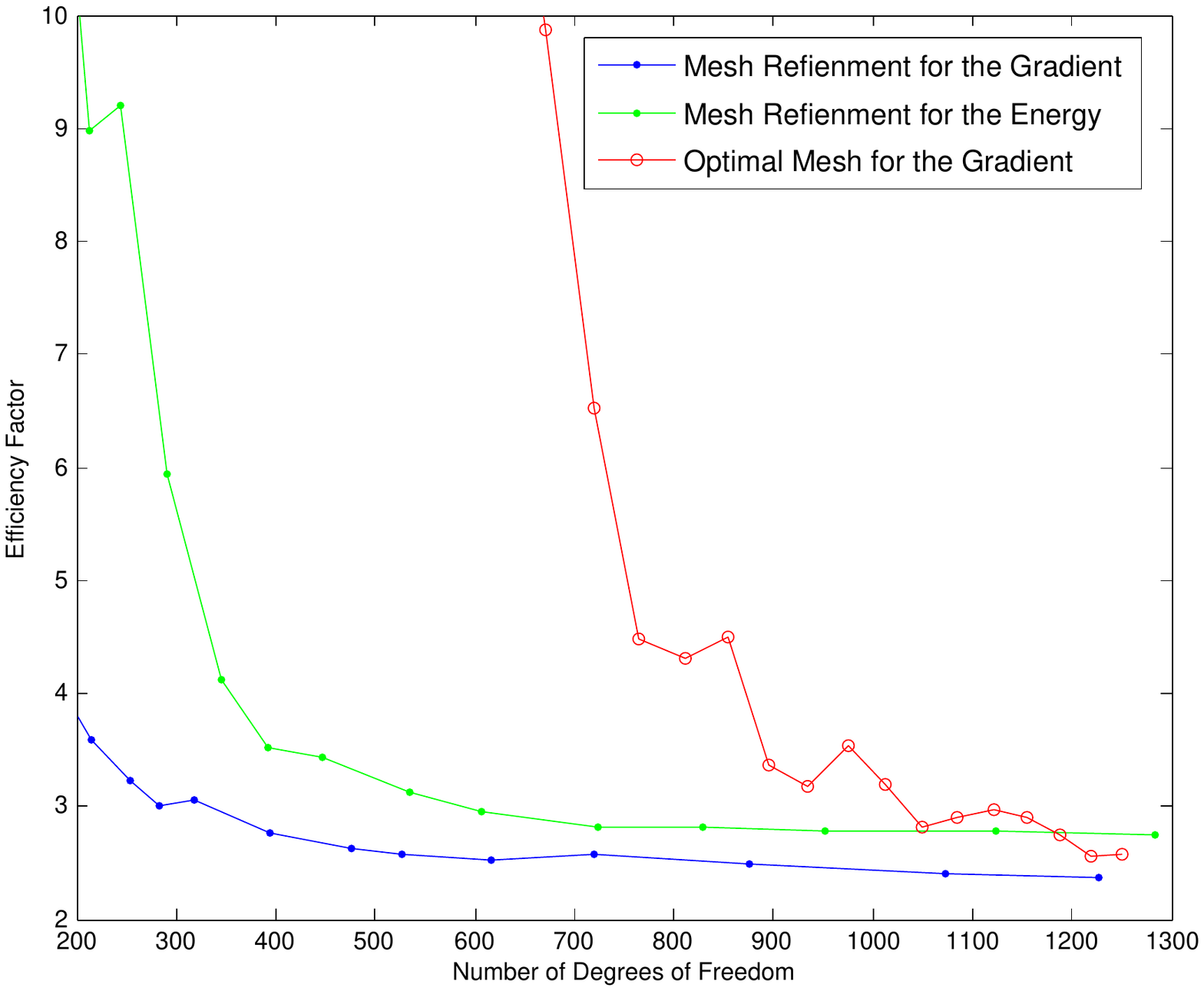}  
      \caption{Efficiency Factor of the Energy} 
      \label{Fig:Efficiency_Factor_Energy}
    \end{center}
  \end{minipage}
\end{figure}

Figure \ref{Fig:Relative_Error_in_Energy} shows that the refinement
based on the energy error performs the best among all the three mesh generating schemes. 

Figure \ref{Fig:Efficiency_Factor_Energy} shows the efficiency
factor of the error estimator of the energy. For the same reason, this
factor decreases as the
number of degrees of freedom increases and finally becomes stable.

\section{Conclusion}
We have presented the a posteriori error estimates for the Consistent
Energy-Based QC method in one dimension. The procedure of the estimate
is the same as that in \cite{Wang:2010a}. However, since the
formulation of the QC problem is newly developed and is totally
different from previous ones, new techniques have been developed and
applied to deal with the difficulty in the analysis. Several results
derived may be of independent interest and usefulness. In addition,
the error estimate of the total energy is also derived. Numerical
experiments are also implemented to illustrate our analysis.

Particular interesting future work are the extension and the
implementation of the a posteriori error estimate in higher
dimensional problems. The difficulty lies in the complication of the
formulation and the varied location of the
interaction bonds. However, since a priori analysis for the two
dimensional problem has been proposed
\cite{OrtShap2010b}, ways of circumventing these difficulties could be
a source of reference.

\appendix
\section{Detailed Analysis for the Residuals of the Stored Energy}
In this section, we provide the omitted detailed analysis for the
residuals of the stored energy, namely
\begin{displaymath}
\Es'_\a(J_{\Us_{\qc}} y_h)[v]-\Es'_{\qc}(y_h)[J_\Us [v] \text{
  and } \Es_\a(J_{\Us_{\qc}} y_h) - \Es_{\qc}(y_h),
\end{displaymath} 
where $y_h \in \Ys_{\qc}$, $y'_h(x) > 0 \ \forall x \in \R$ and $v \in
\Us$.

The idea is to find the differences defined by
\begin{equation}
\label{Eq:ResGradStoreEachBond}
\eps \phi'(r_b D_b  y_h) r_bD_b v -|b\cap \Omega_\a| \phi'(r_b D_{b \cap \Omega_\a} y_h)  D_{b \cap
  \Omega_\a} v -\frac{1}{r_b} \int_{b \cap
  \Omega_c} \phi'(\nabla_{r_b} y_h) \nabla_{r_b} v \dx,
\end{equation}
and 
\begin{equation}
\label{Eq:ResEnStoreEachBond}
\eps \phi(r_b D_b y_h) - |b \cap \Omega_\a|{r_b} \phi\big(r_b D_{b \cap
  \Omega_\a} y \big) - \frac{1}{r_b} \int_{b\cap \Omega_c}
\phi(\nabla_{r_b}y(x))\dx,
\end{equation}
for each interaction bond $b$.

We have analyzed the cases that $b \in \Omega_\a$ and $b \in T_k \cap
\Omega_c$ and are left with the analysis for the cases that $b$ is
across the atomistic-continuum interface and the boundaries of the
elements in the continuum region. There are three cases and in each
case there are three subcases to be considered.

Case 1: $b$ is across two adjacent elements $T_k, T_{k+1} \in
\Omega_c$. In this case $|b \cap \Omega_\a| = 0$ and the atomistic
contribution of the interaction bond in the QC energy is $0$.

Subcase 1: If $b = \big(\ell_k \eps, (\ell_k+1) \eps \big)$, then $r_b = 1$, $b \cap T_k = [\ell_k\eps, x^h_k]$, $b \cap
T_{k+1} = [x^h_k, (\ell_k+1)\eps]$, $r_bD_bv = v'_{\ell_k+1}$ and 
\begin{displaymath}
r_b D_b y_h = \frac{y_h((\ell_k+1)\eps) - y_h(\ell_k\eps)}{\eps} = (1-\theta_k) y_h'|_{T_{k+1}} + \theta_k y_h'|_{T_k}.
\end{displaymath}
We have
\begin{align*}
&\eps \phi'(r_b D_b  y_h) r_bD_b v -\frac{1}{r_b} \int_{b \cap
  \Omega_c} \phi'(\nabla_{r_b} y_h) \nabla_{r_b} v \dx 
\nonumber\\
= & \eps \phi' \big( (1-\theta_k) y_h'|_{T_{k+1}} + \theta_k
y_h'|_{T_k} \big) v'_{\ell_k+1} - \frac{1}{r_b} \phi'(r_b
y_h|_{T_{k+1}}) \int_{b \cap T_k} r_b v' \dx - \frac{1}{r_b} \phi'(r_b
y_h|_{T_{k}}) \int_{b \cap T_{k+1}} r_b v' \dx 
\nonumber\\
= & \eps \bigg\{ \phi'(r_b D_b  y_h) - (1-\theta_k)
  \phi'(y_h'|_{T_{k+1}}) - \theta_k\phi'(\theta_k y_h'|_{T_k}) \bigg\}v'_{\ell_k+1},
\end{align*}
and
\begin{align}
&\eps \phi(r_b D_b y_h) - \frac{1}{r_b} \int_{b\cap \Omega_c}
\phi(\nabla_{r_b}y(x))\dx 
\nonumber\\
=& \eps\phi((1-\theta_k) y_h'|_{T_{k+1}} + \theta_k
y_h'|_{T_k}) - \int_{b \cap T_{k}} \phi(y_h'|_{T_k})\dx -\int_{b
  \cap T_{k+1}} \phi(y_h'|_{T_{k+1}})\dx \nonumber\\
=& \eps \bigg( \phi((1-\theta_k) y_h'|_{T_{k+1}} + \theta_k
y_h'|_{T_k}) - \theta_k \phi(y_h'|_{T_k}) - (1-\theta_k)
\phi(y_h'|_{T_{k+1}}) \bigg).
\end{align}

Subcase 2: If $b = \big((\ell-1)_k \eps, (\ell_k+1) \eps \big)$, then $r_b
= 2$, $b \cap T_k = \big[(\ell_k-1)\eps, x^h_k\big]$, $b \cap T_{k+1} =
\big[x^h_k, (\ell_k+1)\eps\big]$, $r_bD_bv = v'_{\ell_k+1} + v'_{\ell_k}$ and 
\begin{displaymath}
r_b D_b y_h = \frac{y_h(\ell_k\eps) - y_h((\ell_k-1)\eps)}{\eps} = (1-\theta_k) y_h'|_{T_{k+1}} + (1+\theta_k) y_h'|_{T_k}.
\end{displaymath}
We have
\begin{align*}
&\eps \phi'(r_b D_b  y_h) r_bD_b v -\frac{1}{r_b} \int_{b \cap
  \Omega_c} \phi'(\nabla_{r_b} y_h) \nabla_{r_b} v \dx 
\nonumber\\
=&\eps \phi'\big(  (1-\theta_k) y_h'|_{T_{k+1}} + (1+\theta_k)
y_h'|_{T_k} \big) (v'_{\ell_k+1} + v'_{\ell_k}) 
\nonumber\\
&- \frac{1}{r_b} \phi'(r_b
y_h|_{T_{k+1}}) \int_{b \cap T_k} r_b v' \dx - \frac{1}{r_b} \phi'(r_b
y_h|_{T_{k}}) \int_{b \cap T_{k+1}} r_b v' \dx 
\nonumber\\
=&\eps \bigg\{ \big[\phi'\big(  (1-\theta_k) y_h'|_{T_{k+1}} + (1+\theta_k)
y_h'|_{T_k} \big) - \phi'(2y'_h|_{T_k}) \big] v'_{\ell_k} \nonumber \\
&+ \big[\phi'\big(  (1-\theta_k) y_h'|_{T_{k+1}} + (1+\theta_k)
y_h'|_{T_k} \big) - (1-\theta_k) \phi'(2y'_h|_{T_{k+1}}) - \theta_k
\phi'(2y'_h|_{T_{k}})\big]v'_{\ell_k+1} 
\bigg\},
\end{align*}
and
\begin{align}
&\eps \phi(r_b D_b y_h) - \frac{1}{r_b} \int_{b\cap \Omega_c}
\phi(\nabla_{r_b}y(x))\dx 
 \nonumber\\
=& \frac{1}{2}\eps \bigg( 2\phi((1-\theta_k) y_h'|_{T_{k+1}} +
(1+\theta_k) y_h'|_{T_{k}}) - (1+\theta_k) \phi(2y_h'|_{T_k}) - (1-\theta_k)
\phi(2y_h'|_{T_{k+1}}) \bigg).
\end{align}

Subcase 3: If $b = \big(\ell_k \eps, (\ell_k+2) \eps \big)$, then $r_b
= 2$, $b \cap T_k = \big[\ell_k\eps, x^h_k\big]$, $b \cap T_{k+1} =
\big[x^h_k, (\ell_k+2)\eps\big]$, $r_bD_bv = v'_{\ell_k+2} + v'_{\ell_k+1}$ and 
\begin{displaymath}
r_b D_b y_h = \frac{y_h((\ell_k+2)\eps) - y_h(\ell_k\eps)}{\eps} = (2-\theta_k) y_h'|_{T_{k+1}} + \theta_k y_h'|_{T_k}.
\end{displaymath}
We have
\begin{align*}
&\eps \phi'(r_b D_b  y_h) r_bD_b v -\frac{1}{r_b} \int_{b \cap
  \Omega_c} \phi'(\nabla_{r_b} y_h) \nabla_{r_b} v \dx 
\nonumber\\
=&\eps \phi'\big(  (2-\theta_k) y_h'|_{T_{k+1}} + \theta_k y_h'|_{T_k}) (v'_{\ell_k+2} + v'_{\ell_k+1}) 
\nonumber\\
&- \frac{1}{r_b} \phi'(r_b
y_h|_{T_{k+1}}) \int_{b \cap T_k} r_b v' \dx - \frac{1}{r_b} \phi'(r_b
y_h|_{T_{k}}) \int_{b \cap T_{k+1}} r_b v' \dx 
\nonumber\\
=&\eps \bigg\{ \big[\phi'\big(  (2-\theta_k) y_h'|_{T_{k+1}} +
\theta_k y_h'|_{T_k} \big)  - (1-\theta_k)\phi'(2y'_h|_{T_{k+1}}) - \theta_k\phi'(2y'_h|_{T_k}) \big] v'_{\ell_k+1} \nonumber \\
&+ \big[\phi'\big(  (2-\theta_k) y_h'|_{T_{k+1}} + \theta_k
y_h'|_{T_k} \big) - \phi'(2y'_h|_{T_{k+1}})\big]v'_{\ell_k+2} \bigg\},
\end{align*}
and
\begin{align}
&\eps \phi(r_b D_b y_h) - \frac{1}{r_b} \int_{b\cap \Omega_c}
\phi(\nabla_{r_b}y(x))\dx 
 \nonumber\\
=& \frac{1}{2}\eps \bigg( 2\phi((2-\theta_k) y_h'|_{T_{k+1}} +
\theta_k y_h'|_{T_{k}})  - \theta_k
\phi(2y_h'|_{T_{k+1}}) - (2-\theta_k) \phi(2y_h'|_{T_k}) \bigg).
\end{align}

Case 2: $b$ is across the {\it left} atomistic-continuum interface of an
atomistic region. 

Subcase 1: If $b = (\ell_k \eps, \ell_{k+1} \eps)$, then $r_b = 1$, $b \cap \Omega_c = (\ell_k\eps, x^h_k)$,
$b \cap \Omega_\a = \big(x^h_k, (\ell_{k}+1)\eps\big)$, $r_bD_bv = v_{\ell_k+1}'$
and
\begin{displaymath}
r_bD_by_h = (1-\theta_k)y_h|_{T_{k+1}} + \theta_k y'_h|_{T_k}.
\end{displaymath}
We have
\begin{align}
&\eps \phi'(r_b D_b  y_h) r_bD_b v -|b\cap \Omega_\a| \phi'(r_b D_{b \cap \Omega_\a} y_h)  D_{b \cap
  \Omega_\a} v -\frac{1}{r_b} \int_{b \cap
  \Omega_c} \phi'(\nabla_{r_b} y_h) \nabla_{r_b} v \dx
\nonumber\\
=& \eps \bigg[ \phi'\big( (1-\theta_k)y'_h|_{T_{k+1}} + \theta_k
y'_h|_{T_k}\big) - (1-\theta_k)\phi'(y'_h|_{T_{k+1}}) - \theta_k
\phi'(y'_h|_{T_{k}}) \bigg] v'_{\ell_k+1},
\end{align}
and
\begin{align}
&\eps \phi(r_b D_b y_h) - |b \cap \Omega_\a|{r_b} \phi\big(r_b D_{b \cap
  \Omega_\a} y \big) - \frac{1}{r_b} \int_{b\cap \Omega_c}
\phi(\nabla_{r_b}y(x))\dx
\nonumber\\
= &\eps \big( \phi( (1-\theta_k) y'_h|_{T_{k+1}} + \theta_k
y'_h|_{T_k}) - \theta_k\phi( y_h'|_{T_{k}}) -(1-\theta_k) \phi(
y_h'|_{T_{k+1}}) \big)
\end{align}

Subcase 2: If $b = \big((\ell_k-1) \eps, (\ell_{k}+1) \eps \big)$,
then $r_b = 2$, $b \cap \Omega_c = \big((\ell_k-1)\eps, x^h_k)$,
$b \cap \Omega_\a =\big (x^h_k, (\ell_{k}+1)\eps\big)$, $r_bD_bv = v_{\ell_k+1}'+v_{\ell_k}'$
and
\begin{displaymath}
r_bD_by_h = (1-\theta_k)y'_h|_{T_{k+1}} + (1+\theta_k) y'_h|_{T_k},
\quad r_b D_{b \cap \Omega_\a} y_h = 2 y'_h|_{T_{k+1}} 
\end{displaymath}
We have
\begin{align}
&\eps \phi'(r_b D_b  y_h) r_bD_b v -|b\cap \Omega_\a| \phi'(r_b D_{b \cap \Omega_\a} y_h)  D_{b \cap
  \Omega_\a} v -\frac{1}{r_b} \int_{b \cap
  \Omega_c} \phi'(\nabla_{r_b} y_h) \nabla_{r_b} v \dx
\nonumber\\
=& \eps\bigg\{ \big[ \phi'\big( (1-\theta_k)y'_h|_{T_{k+1}} + (1+ \theta_k)
y'_h|_{T_k}\big) - (1-\theta_k)\phi'(2y'_h|_{T_{k+1}}) - \theta_k
\phi'(2y'_h|_{T_{k}})\big]  v'_{\ell_k} \nonumber\\
 &+\big[ \phi'\big( (1-\theta_k)y'_h|_{T_{k+1}} + (1+ \theta_k)
y'_h|_{T_k}\big) - \phi'(2y'_h|_{T_{k}})\big] v'_{\ell_k}  - \theta_k
\phi'(y'_h|_{T_{k}})  v'_{\ell_k+1},
\end{align}
and
\begin{align}
&\eps \phi(r_b D_b y_h) - |b \cap \Omega_\a|{r_b} \phi\big(r_b D_{b \cap
  \Omega_\a} y \big) - \frac{1}{r_b} \int_{b\cap \Omega_c}
\phi(\nabla_{r_b}y(x))\dx \nonumber\\
=&2\eps \big( \phi( (1-\theta_k) y_h'|_{T_{k+1}} + (1+\theta_k)\theta_k y_h'|_{T_k}) 
- (1-\theta_k) \phi(2y_h'|_{T_{k+1}}) - (1+\theta_k)\phi(2y_h'|_{T_{k}}) \big).
\end{align}

Subcase 3: If $b = (\ell_k \eps, (\ell_{k}+2) \eps)$,
then $r_b = 2$, $b \cap \Omega_c = \big(\ell_k\eps, x^h_k)$,
$b \cap \Omega_\a = \big(x^h_k, (\ell_{k}+2)\eps\big)$, $r_bD_bv =
v_{\ell_k+2}'+v_{\ell_k+1}'$, $D_{b \cap \Omega_\a} v =
\frac{1}{2-\theta_k} v'_{\ell_k+2} + \frac{1-\theta_k}{2-\theta_k}v'_{\ell_k+1}$,
and
\begin{displaymath}
r_bD_by_h = y'_h|_{T_{k+2}} + (1-\theta_k)y'_h|_{T_{k+1}} + \theta_k
y'_h|_{T_k} \quad r_b D_{b \cap \Omega_\a} y_h = \frac{2}{2-\theta_k}
y'_h|_{T_{k+2}} + \frac{2(1-\theta_k)}{2-\theta_k} y'_h|_{T_{k+1}}. 
\end{displaymath}
We have
\begin{align}
&\eps \phi'(r_b D_b  y_h) r_bD_b v -|b\cap \Omega_\a| \phi'(r_b D_{b \cap \Omega_\a} y_h)  D_{b \cap
  \Omega_\a} v -\frac{1}{r_b} \int_{b \cap
  \Omega_c} \phi'(\nabla_{r_b} y_h) \nabla_{r_b} v \dx
\nonumber\\
=& \eps\bigg\{ \big[ \phi'\big( y'_h|_{T_{k+2}} +
(1-\theta_k)y'_h|_{T_{k+1}} + \theta_ky'_h|_{T_k}\big) - \phi'(\frac{2}{2-\theta_k}
y'_h|_{T_{k+2}} + \frac{2(1-\theta_k)}{2-\theta_k} y'_h|_{T_{k+1}})\big]  v'_{\ell_k+2} \nonumber\\
 &+\big[ \phi'\big( y'_h|_{T_{k+2}} +
(1-\theta_k)y'_h|_{T_{k+1}} + \theta_ky'_h|_{T_k}\big) - (1-\theta_k) \phi'(\frac{2}{2-\theta_k}
y'_h|_{T_{k+2}} + \frac{2(1-\theta_k)}{2-\theta_k} y'_h|_{T_{k+1}})
\nonumber\\
&\quad \ -\theta_k \phi'(2y'_h |_{T_k}) \big] v'_{\ell_k+1},
\end{align}

Case 2: $b$ is across the {\it right} atomistic-continuum interface of an
atomistic region. 

Subcase 1: If $b = (\ell_k \eps, \ell_{k+1} \eps)$, then $r_b = 1$, $b
\cap \Omega_c = \big(x^h_k, (\ell_{k}+1)\eps\big)$,
$b \cap \Omega_\a = (\ell_k\eps, x^h_k)$, $r_bD_bv = v_{\ell_k+1}'$
and
\begin{displaymath}
r_bD_by_h = (1-\theta_k)y_h|_{T_{k+1}} + \theta_k y'_h|_{T_k}.
\end{displaymath}
We have 
\begin{align}
&\eps \phi'(r_b D_b  y_h) r_bD_b v -|b\cap \Omega_\a| \phi'(r_b D_{b \cap \Omega_\a} y_h)  D_{b \cap
  \Omega_\a} v -\frac{1}{r_b} \int_{b \cap
  \Omega_c} \phi'(\nabla_{r_b} y_h) \nabla_{r_b} v \dx
\nonumber\\
=& \eps \bigg[ \phi'\big( (1-\theta_k)y'_h|_{T_{k+1}} + \theta_k
y'_h|_{T_k}\big) - (1-\theta_k)\phi'(y'_h|_{T_{k+1}}) - \theta_k
\phi'(y'_h|_{T_{k}}) \bigg] v'_{\ell_k+1},
\end{align}
and
\begin{align}
&\eps \phi(r_b D_b y_h) - \frac{1}{r_b} \int_{b\cap \Omega_c}
\phi(\nabla_{r_b}y(x))\dx 
\nonumber\\
=& \eps\phi((1-\theta_k) y_h'|_{T_{k+1}} + \theta_k
y_h'|_{T_k}) - \int_{b \cap T_{k}} \phi(y_h'|_{T_k})\dx -\int_{b
  \cap T_{k+1}} \phi(y_h'|_{T_{k+1}})\dx \nonumber\\
= &\eps \big( \phi( (1-\theta_k) y'_h|_{T_{k+1}} + \theta_k
y'_h|_{T_k}) - \theta_k\phi( y_h'|_{T_{k}}) -(1-\theta_k) \phi( y_h'|_{T_{k+1}}) \big).
\end{align}

Subcase 2: If $b = \big(\ell_k \eps, (\ell_{k}+2) \eps\big)$, then $r_b = 2$, $b
\cap \Omega_c = \big(x^h_k, (\ell_{k}+2)\eps\big)$,
$b \cap \Omega_\a = (\ell_k\eps, x^h_k)$, $r_bD_bv = v'_{\ell_k+2} +
v'_{\ell_k+1}$, $r_bD_{b\cap\Omega_\a} v = v'_{\ell_k+1}$ and
\begin{align}
r_bD_by_h = (2-\theta_k)y'_h|_{T_{k+1}} + \theta_k y'_h|_{T_k} \text{
  and } r_bD_{b\cap \Omega_\a}y_h = y'_h|_{T_k}.
\end{align}
We have
\begin{align}
&\eps \phi'(r_b D_b  y_h) r_bD_b v -|b\cap \Omega_\a| \phi'(r_b D_{b \cap \Omega_\a} y_h)  D_{b \cap
  \Omega_\a} v -\frac{1}{r_b} \int_{b \cap
  \Omega_c} \phi'(\nabla_{r_b} y_h) \nabla_{r_b} v \dx
\nonumber\\
=& \eps \bigg\{ \big[ \phi'\big((2-\theta_k)y'_h|_{T_{k+1}} +
\theta_ky'_h|_{T_k}\big) - \theta_k\phi'\big(2y'_h|_{T_k}) -
(1-\theta_k)\phi'\big(2y'_h|_{T_{k+1}} \big] v'_{\ell_k+1} \nonumber\\
&+\big[ \phi'\big((2-\theta_k)y'_h|_{T_{k+1}} +
\theta_ky'_h|_{T_k}\big) -\phi'(2y'_h|_{T_{k+1}}) \big] v'_{\ell_k+2},
\end{align}
and
\begin{align}
&\eps \phi(r_b D_b y_h) - \frac{1}{r_b} \int_{b\cap \Omega_c}
\phi(\nabla_{r_b}y(x))\dx 
\nonumber\\
=&\eps \big( \phi( (1-\theta_k) y'_h|_{T_{k+1}} + \theta_k
y'_h|_{T_k}) - \theta_k\phi( y_h'|_{T_{k}}) -(1-\theta_k) \phi( y_h'|_{T_{k+1}}) \big).
\end{align}

Subcase 3: If $b = \big((\ell_k-1) \eps, (\ell_{k}+1) \eps\big)$, then $r_b = 2$, $b
\cap \Omega_c = \big(x^h_k, (\ell_{k}+1)\eps\big)$,
$b \cap \Omega_\a = \big((\ell_k-1)\eps, x^h_k\big)$, $r_bD_bv = v'_{\ell_k+1} +
v'_{\ell_k}$, $r_bD_{b\cap\Omega_\a} v = \frac{\theta_k}{1+\theta_k}v'_{\ell_k+1}+\frac{1}{1+\theta_k}v'_{\ell_k}$ and
\begin{align}
r_bD_by_h = (1-\theta_k)y'_h|_{T_{k+1}} + \theta_k y'_h|_{T_k} +
y'_h|_{T_{k-1}} \text{ and } r_bD_{b\cap \Omega_\a}y_h =
\frac{2\theta_k}{1+\theta_k}y'_h|_{T_k} + \frac{2}{1+\theta_k} y'_h|_{T_{k-1}}.
\end{align}
We have
\begin{align}
&\eps \phi'(r_b D_b  y_h) r_bD_b v -|b\cap \Omega_\a| \phi'(r_b D_{b \cap \Omega_\a} y_h)  D_{b \cap
  \Omega_\a} v -\frac{1}{r_b} \int_{b \cap
  \Omega_c} \phi'(\nabla_{r_b} y_h) \nabla_{r_b} v \dx
\nonumber\\
=& \eps \bigg\{ \big[ \phi'\big((1-\theta_k)y'_h|_{T_{k+1}} +
\theta_ky'_h|_{T_k} + y'_h|_{T_{k-1}}\big)
-\phi'(\frac{2\theta_k}{1+\theta_k}y'_h|_{T_{k}} +
\frac{2}{1+\theta_k}y'_h|_{T_{k-1}}) \big] v'_{\ell_k} \nonumber\\
&+\big[ \phi'\big((1-\theta_k)y'_h|_{T_{k+1}} +
\theta_ky'_h|_{T_k} + y'_h|_{T_{k-1}}\big) - \theta_k\phi'(\frac{2\theta_k}{1+\theta_k}y'_h|_{T_{k}} +
\frac{2}{1+\theta_k}y'_h|_{T_{k-1}}) \nonumber\\
&-(1-\theta_k) \phi'(y'_h|_{T_{k+1}}) \big] v'_{\ell_k+1},
\end{align}
and
\begin{align}
&\eps \phi(r_b D_b y_h) - \frac{1}{r_b} \int_{b\cap \Omega_c}
\phi(\nabla_{r_b}y(x))\dx 
\nonumber\\
=&2\eps \big( \phi( (1-\theta_k) y_h'|_{T_{k+1}} + \theta_k y_h'|_{T_k} +
y_h'|_{T_{k-1}}) \nonumber\\
&- (1+\theta_k) \phi\big( \frac{2\theta_k}{1+\theta_k} y_h'|_{T_k} +
\frac{2}{1+\theta_k}y_h'|_{T_{k-1}}\big) - (1-\theta_k)\phi(2y_h'|_{T_{k+1}}) \big).
\end{align}

\section{Approximation Properties}
\label{AppendixA}
In this section, we prove some approximation properties which we have
used but are hardly found in standard text books.

\begin{lemma}
\label{lemmaA1}
Let $v \in C^0(\R) \cap W^{1,2}(\R)$ be a periodic function with
$[a,b]$ being one of its period. Let $v_h$ be a $\Ps_1$ interpolation of $v$
with respect to the nodes $a \le
x_0 < x_1 < \cdots < x_n \le b \le
x_{n+1} = x_0 + (b-a)$ in $[x_0,x_{n+1}]$, subject
to a constant, i.e., $v_h (x_k) = v(x_k)+C$, for $k =
\{1,2,\ldots, n+1 \}$,  and is extended periodically with period $b-a$. Then the
following estimate holds:
\begin{equation}
\|v_h'\|_{L^2_{[a,b]}} \le \|v'\|_{L^2_{[a,b]}},
\end{equation} 
where $v'$ and $v_h'$ denote the weak derivatives of $v$ and $v_h$ respectively.
\end{lemma}

\begin{proof}
First we note that, since $v \in C^0(\R)$ and $v_h$ is a $\Ps_1$
interpolation of $v$,
the weak derivative of $v_h$ on $(x_{k}, x_{k+1})$ is defined by
\begin{equation*}
v_h' (x)=
\frac{v(x_{k+1})-v(x_{k})}{x_{k+1}-x_{k}}.
\end{equation*}
Since $v \in C^0(\R)$ is piecewise differentiable, we have 
\begin{equation*}
v(x_{k+1})-v(x_{k}) = \int_{x_{k}}^{x_{k+1}}
v'(t) \dt,
\end{equation*}
where $v'$ is the weak derivative of $v$.
By the periodicity of $v_h'$ and $v'$, and Cauchy-Schwarz Inequality, we have
\begin{align*}
\|v_h'\|^2_{L^2_{[a,b]}} 
& = \int_a^b \big[v_h'(x)\big]^2 \dx \\
& = \int_ {x_0}^{x_{n+1}} \big[ v_h'(x) \big]^2\dx\\
& = \sum_{k = 0}^{n}
\int_{x_{k}}^{x_{k+1}}  (\frac{v(x_{k+1})-v(x_{k})}{x_{k+1}-x_{k}})^2
\dx\\
& = \sum_{k = 0}^{n} \int_{x_{k}}^{x_{k+1}} \frac{(\int_{x_{k}}^{x_{k+1}}
v'(t) \dt)^2}{(x_{k+1}-x_{k})^2} \dx\\
& \le \sum_{k = 0}^{n} \frac{1}{x_{k+1}-x_{k}}
(\int_{x_{k}}^{x_{k+1}} \dt)
(\int_{x_{k}}^{x_{k+1}}|v'(t)|^2\dt)\\
& = \sum_{k = 0}^{n}\int_{x_{k}}^{x_{k+1}}|v'(t)|^2\dt\\
& = \int_a^b |v'(t)|^2\dt\\
& = \|v'\|^2_{L^2_{[a,b]}}.
\end{align*}
Taking the square root on both sides gives the stated result.
\end{proof}

\begin{lemma}
\label{lemmaA2}
Let $v \in C^0([a,b]) \cap W^{1,2}([a,b])$ and $I_hv$ is the $\Ps_1$
function that interpolates $v$ at the points $a$ and $b$. We have the
following inequality:
\begin{equation}
\|v'-(I_hv)'\|^2_{L^2_{(a,b)}} \le  \|v'\|^2_{L^2_{(a,b)}}.
\end{equation}
\end{lemma}

\begin{proof}
Since $v(a) = I_hv(a)$ and $v(b) =
I_hv(b)$, by the definition of $I_hv$, we have
\begin{equation*}
\int_a^bv'\dx = \int_a^b (I_hv)' dx,
\end{equation*}
and equivalently,
\begin{equation*}
\int_a^b\big(v'-(I_hv)'\big) \cdot 1 \dx = 0,
\end{equation*}
where $v'$ denotes the weak derivative of $v$ on $[a,b]$. This shows
that $(I_hv)'$ is the best $L^2$ approximation of $v'$ in
the space of $\Ps_0$ functions as $(I_hv)'$ is a constant. Therefore, by the
property of best approximation,
\begin{equation}
\|v'-(I_hv)'\|^2_{L^2_{(a,b)}} \le  \|v'-C\|^2_{L^2_{(a,b)}},
\end{equation}
for any constant $C$. In particular, if we choose $C$ to be $0$, the
stated result holds.
\end{proof}

\section{Discrete Sobolev Inequalities on Non-uniform mesh}
\label{AppendixB}
In this section, we prove some discrete Soblev inequalities on
non-uniform mesh that are used in the residual analysis for the
external force. These results are extensions to the inequalities proved in
\cite[Lemma A.1, Lemma A.2, Theorem A.4]{Ortner:2008a} on non uniform
mesh.

\begin{lemma}
\label{Lemma:PrePoincare}
Let $\gb \in \R^L$, $\epsb^0, \epsb^1 \in \R^L$ and $\eps^0_i, \eps^1_i>0 \  \forall i = 1,
\ldots, L$, $\gb' = (g'_i)_{i=2}^L\in \R^{L-1}$, $g_i' := \frac{g_i-g_{i-1}}{\eps^1} \ i = 2,\ldots,L$. If $\sum_{i=1}^L\eps^0_i g_i = 0$, then 
\begin{equation}
|g_i| \le \frac{1}{h} \sum_{i=2}^L\bar{\eps}^1_k |g_k'| \phi_{i,k},
\end{equation}
where, $h = \sum_{i=1}^L\eps^0_i$, $\phi_{i,k} = \sum_{\ell=1}^{k-1}\eps^0_{\ell}$ for $k =
2,\ldots,i$ and  $\phi_{i,k} = \sum_{\ell=k}^{L}\eps^0_{\ell}$ for $k =i+1,\ldots,L$.
\end{lemma}

\begin{proof}
Let $i \in \{1,\ldots,L\}$, then
\begin{align*}
h|g_i| & = |hg_i - \sum_{j=1}^{L} \eps^0_j g_j|\\
          & = |\sum_{j=1}^{L} \eps^0_j g_i - \sum_{j=1}^{L} \eps^0_j g_j|\\
          & \le \sum_{j=1}^{i-1}\eps^0_j|g_i-g_j|+\sum_{j=i+1}^{L}\eps^0_j |g_i-g_j|.
\end{align*}
Since 
\begin{equation*}
|g_i-g_j| = |\sum_{k=j+1}^i \eps^1_k g'_k|,
\end{equation*}
we have
\begin{align*}
h|g_i|  &\le \sum_{j=1}^{i-1} \eps^0_j \sum_{k = j+1}^i \eps^1_k
|g'_k|+ \sum_{j=i+1}^{L} \eps^0_j \sum_{k = i+1}^j \eps^1_k|g'_k|\\
         &= \sum_{k=2}^i  \eps^1_k |g'_k| \big(\sum_{j= 1}^{k-1}\eps^0_j\big) + \sum_{k=i+1}^L  \eps^1_k |g'_k| \big(\sum_{j= k}^{L}\eps^0_j\big)\\
         &= \sum_{k=2}^{L}\eps^1_k|g'_k|\phi_{i,k}.          
\end{align*} 
Divide both sides by $h$, we obtain the stated result.
\end{proof}

\begin{lemma}
\label{Lemma:DiscPoincareNonUniMesh}
(Discrete Poincare's Inequality) Suppose that $L \ge 1$, $\epsb^0, \epsb^1 \in \R^L$ with $\eps^0_i, \eps^1_i > 0$, $\forall i =1, \ldots, L $. 
Let $\gb \in \R^{L}$ such that  $\sum_{i=1}^L\eps^0_i g_i = 0$ and $\gb' = (g'_i)_{i=2}^L \in \R^{L-1}$ such that $g_i' =
\frac{g_i-g_{i-1}}{\eps^1_i}$. Define $\mathcal{D}_0$ to be the set $\{1, \ldots,
L\}$ and $\mathcal{D}_1$ to be the set $\{2, \ldots, L\}$,
then
\begin{equation}
\Vert \gb \Vert _{\ell^{p}_{\epsb^0 (\mathcal{D}_0)}} \le \frac{1}{2} \frac{L^2 \max\{\max_{1 \le i \le L} \eps^0_i, \max_{2 \le k \le L} \eps^1_k\}^2}{h}
 \Vert \gb' \Vert _{\ell^{p}_{\epsb^1(\mathcal{D}_1)}},
\end{equation}
for $p \in \{1, \infty\}$, where $\ h = \sum_{i=1}^L\eps^0_i$.
\end{lemma}

\begin{proof}
Using the result of Lemma \ref{Lemma:PrePoincare}, we have
\begin{align*}
\sum_{i=1}^L \eps^0_i |g_i| 
&\le \sum_{i=1}^L \frac{\eps^0_i}{h}
\sum_{k=2}^i \eps^1_k |g_k'| \phi_{i,k} +\sum_{i=1}^L \frac{\eps^0_i}{h}
\sum_{k=i+1}^L \eps^1_k |g_k'| \phi_{i,k}\\
& = \frac{1}{h} \bigg[ \sum_{k=2}^L \big(\sum_{i=1}^L \eps^0_i \phi_{i,k}\big) \eps^1_k |g_k'| \bigg].
\end{align*}
Since 
\begin{displaymath}
\sum_{i = 1}^L \eps^0_i \phi_{i,k} \le \max_{1\le i \le L} \eps^0_i
\sum_{i = 1}^L \phi_{i,k} = \max_{1 \le i \le L} \eps^0_i \bigg[ \sum_{i=1}^{k-1}\phi_{i,k} + \sum_{i=k}^{L}\phi_{i,k} \bigg],
\end{displaymath}
and
\begin{align*}
\sum_{i=1}^{k-1}\phi_{i,k} + \sum_{i=k}^{L}\phi_{i,k} 
& \le (k-1)
\sum_{\ell=k}^L\eps^0_{\ell} +
\big(L-(k-1)\big)\sum_{\ell=1}^{k-1}\eps^0_{\ell}\\
& \le \big[(k-1) \big(L-(k-1)\big) +\big(L-(k-1)\big)  (k-1)\big]
\max_{1 \le i \le L} \eps^0_i \\
& \le \frac{1}{2} \max_{1 \le i \le L} \eps^0_i L^2.
\end{align*}
Put these results together, we obtain the stated result for $p = 1$.
For $p = \infty$, 
\begin{align*}
|g_i| &\le \frac{1}{h} \sum_{k = 2}^{L} \eps^1_k |g_k'|
\phi_{i,k}\\
& \le \frac{1}{h} \bigg[ \sum_{k=2}^i \eps^1_k |g_k'|
\phi_{i,k} + \sum_{k=i+1}^L \eps^1_k |g_k'|
\phi_{i,k} \bigg]\\
& \le \frac{1}{h} \sum_{k = 2}
^ L \phi_{i,k} \max_{2 \le k \le L} \eps^1_k |g_k'| \\
& \le \frac{1}{2} \frac{L^2 \max_{1\le i \le L}\eps^0_i}{h} \max_{2 \le k \le L} \eps^1_k |g_k'|.
\end{align*}
The stated result is obtained by taking the maximum of $\eps^0_i$ and $\eps^1_k$ over $\Ds_0$ and $\Ds_1$.
\end{proof}

\begin{lemma}
\label{Lemma:DiscFriedrichNonUniMesh}
(Discrete Friedrichs' Inequality) Suppose that $L \ge 1$, $\epsb^0$, $\epsb^1$,
$\mathcal{D}_0$, $\mathcal{D}_2$ are the same as in Lemma \ref{Lemma:DiscPoincareNonUniMesh}. Let
$\fb \in \R^{L}$ such that $f_1 = f_L=0$, and $\fb' =
(f_i')_{i=2}^{L} \in \R^{L-1}$ such that $f_i' =
\frac{f_i-f_{i-1}}{\eps^1_i}$, then
\begin{equation}
\Vert \fb \Vert _{\ell^{p}_{\epsb^0 (\mathcal{D}_0)}} \le \frac{1}{2} (L-1)
\max_{2\le i \le L-1} \max\{\eps^0_i, \eps^1_i\} \Vert \fb' \Vert _{\ell^{p}_{\epsb^1 (\mathcal{D}_1)}},
\end{equation}
for $p \in \{1, \infty\}$.
\end{lemma}
\begin{proof}
For $p = 1$,
\begin{align*}
\sum_{i = 1}^L \eps^0_i |f_i| 
& = \sum_{i = 2}^{L-1} \eps^0_i |f_i|\\
& = \frac{1}{2} \sum_{i=2}^{L-1} \eps^0_i \big[ |\sum_{j = 2}^i (f_j-f_{j-1})| + |\sum_{j = i+1}^L(f_j-f_{j-1})| \big]\\
& \le \frac{1}{2} \sum_{i=2}^{L-1} \eps^0_i\big[ \sum_{j = 2}^i \eps^1_j |f_j'| + \sum_{j = i+1}^L \eps^1_j |f_j'|\big]\\
& = \frac{1}{2} \sum_{i=2}^{L-1}\eps^0_i \sum_{j = 1}^{L} \eps^1_j |f_j'|\\
&\le \frac{1}{2} (L-1) \max_{2\le i \le L-1} \eps^0_i \sum_{j = 1}^{L} \eps^1_j |f_j'|.
\end{align*}
For $p = \infty$, 
\begin{displaymath}
|f_i| \le \sum_{j = 2}^i \eps^1_j|f_j'| =  (i-1)  \max_{2\le j \le L} \eps^1_j \max_{2\le j \le L} |f_j'|,
\end{displaymath}
and 
\begin{displaymath}
|f_i| \le \sum_{j = i+1}^L \eps^1_j |f_j'| =  (L-i) \max_{2\le j \le L} \eps^1_j \max_{2\le j \le L}|f_j'|.
\end{displaymath}
Thus we have
\begin{align*}
\max_{i \in \mathcal{D}_0} |f_i| &\le \min (i-1, L-i)  \max_{2\le j \le L} \eps^1_j \max_{2\le j \le L}|f_j'| \\
 & \le \frac{1}{2} (L-1) \max_{2\le j \le L} \eps^1_j \max_{2\le j \le L}|f_j'|.
\end{align*}
\end{proof}

\begin{remark}
The bounds we have got here are not optimal as if $\eps_i$'s and
$\bar{\eps}_j$'s vary too much, taking the maximum of them in the
inequalities could significantly reduce the sharpness of the estimate. However, for the
analysis of this paper, such a bound is optimal enough to produce
efficient error estimators and we leave the work of looking for optimal
bounds to future work.
\end{remark}

\begin{theorem}
\label{Theo:BoundOnIntErr}
(bounds on the interpolation error) Let $L \ge 1$, $\epsb^0, \epsb^1,\epsb^2\in \R^{L}$, 
with $\eps^0_i, \eps^1_i, \eps^2_i>0\ \forall i =1, \ldots, L $. Let
$\fb \in \R^L$ and $\Fb = \in \R^L$ such that $F_1 = f_1$ and
\begin{equation}
F_i = f_1 + \frac{\sum_{j=2}^i\eps^0_i}{h}(f_L-f_1) \quad i = 2,
\ldots, L , 
\end{equation}
where $h = \sum_{i=2}^{L}\eps^0_i$. Define $\fb' = (f_i')_{i=2}^{L} \in
\R^{L-1}$ such that $f_i'= \frac{f_i-f_{i-1}}{\eps^1_i}$ and $\fb'' = (f_i'')_{i=2}^{L-1} \in
\R^{L-2}$ such that $f_i''= \frac{f'_{i+1}-f'_{i}}{\eps^2_i}$, and $\Fb'$ and $\Fb''$ are defined
in the same way. Let $\mathcal{D}_0$, $\mathcal{D}_1$ be the same sets defined in Lellmma \ref{Lemma:DiscPoincareNonUniMesh} and
$\mathcal{D}_2$ be the set $\{2, \ldots, L-1\}$. Then, for $p \in
\{1, \infty \}$, 
\begin{equation}
\| \fb - \Fb \|_{\ell_{\epsb^0}^p (\mathcal{D}_0)} \le \frac{1}{4}
\frac{L^3 \max_{2 \le i \le L-1} \eps_i^0 \max_{2 \le j \le L-1}
  \eps^1_j \max_{2 \le k \le L-1}  \eps^2_j}{h} \|\fb''\|_{\ell_{\epsb^2}^p (\mathcal{D}_2)}.
\end{equation}
\end{theorem}

\begin{proof}
Let $\gb = \fb - \Fb$, by the definition of $\Fb$, we have $g_1 = g_L = 0$ and 
\begin{displaymath}
\sum_{i=2}^L\eps_ig_i' = \sum_{i=2}^L(f_i-f_{i-1})
-\sum_{i=2}^L(F_i-F_{i-1}) = 0.
\end{displaymath} 
By Lemma \ref{Lemma:DiscFriedrichNonUniMesh}, 
\begin{displaymath}
\| \gb \|_{\ell_{\epsb^1}^p (\mathcal{D}_0)} \le \frac{1}{2} (L-1)
\max_{2\le i \le L-1} \max\{\eps^0_i, \eps^1_i\} \|\gb'\|_{\ell_{\epsb}^p (\mathcal{D}_1)},
\end{displaymath}
as $g_1 = g_L = 0$, and by Lemma \ref{Lemma:DiscPoincareNonUniMesh}, 
\begin{displaymath}
\|\gb'\|_{\ell_{\epsb}^p (\mathcal{D}_1)} \le \frac{1}{2} \frac{L^2 \max\{\max_{1 \le i \le L} \eps^1_i, \max_{2 \le k \le L-1} \eps^2_k\}^2}{h} \|\gb''\|_{\ell_{\bar{\epsb}}^p (\mathcal{D}_2)},
\end{displaymath}
as $\sum_{i=2}^L\eps_i g_i' = 0$. Since $\Fb''=0$, from which we know
$\gb'' = \fb''$, the stated estimate holds.
\end{proof}
\bibliographystyle{plain}
\bibliography{qc1}
\end{document}